\documentclass[10pt,a4paper]{amsart}
\usepackage{amsmath,amssymb}
\usepackage{fullpage}
\usepackage{enumerate}
\usepackage{hyperref}
\usepackage{xcolor}
\usepackage[all]{xy}

\usepackage[stable]{footmisc}

\theoremstyle{plain}
\newtheorem{theorem}{Theorem}
\newtheorem{corollary}[theorem]{Corollary}
\newtheorem{conjecture}[theorem]{Conjecture}
\newtheorem{proposition}[theorem]{Proposition}
\newtheorem{question}[theorem]{Question}

\theoremstyle{remark}
\newtheorem{definition}[theorem]{Definition}
\newtheorem{remark}[theorem]{Remark}
\newtheorem{example}[theorem]{Example}

\numberwithin{theorem}{section}

\title{Hilbert properties of varieties}
\author{Arno Fehm and Ariyan Javanpeykar}

\begin{document}

\begin{abstract}
This is a survey of results on the Hilbert property of algebraic varieties, and variants of it.    
\end{abstract}

\maketitle

\noindent
Starting from Hilbert's irreducibility theorem,
various Hilbert properties of fields and varieties have been
introduced, studied and applied.
This survey aims to give an overview of the various notions
and to document the state of the art of the research on this topic.
On the way we also collect several important open problems. \smallbreak

\tableofcontents

\subsection*{Where to find what}

Section \ref{sec:one} introduces and tells the history of the central notions of this survey:
Hilbert's irreducibility theorem, Hilbertian fields, the Hilbert property, and the weak Hilbert property.
The remaining two subsections discuss Hilbert properties for integral (instead of rational) points,
and special varieties in the sense of Campana.
Section \ref{sec:preservation} collects general results on the preservation of various Hilbert properties
under for example base change, morphisms and products.
The remaining sections each discuss one class of concrete varieties, usually over number fields, finitely generated fields or Hilbertian fields:
curves in Section \ref{sec:curves},
algebraic groups in Section \ref{sec:algebraic_groups},
surfaces in Section \ref{sec:surfaces},
and some classes of higher-dimensional varieties in Section \ref{sec:higher_dim} --
closed subvarieties of abelian varieties, Fano varieties, varieties with nef tangent bundle, Kummer varieties,
and symmetric products.

\subsection*{Notation}
Let $K$ be a field.
We denote by $\overline{K}$ an algebraic closure of $K$.
A {\em $K$-variety} or a {\em variety over $K$} is an integral separated scheme of finite type over $K$.
If $X$ is a $K$-variety and $L/K$ a field extension, we denote by $X_L=X\times_{{\rm Spec}(K)}{\rm Spec}(L)$ the base change of $X$ to $L$.
If $X$ and $Y$ are $K$-varieties, $X\sim Y$ denotes that $X$ and $Y$ are birationally equivalent.

\subsection*{Disclaimer}
Due to the vast amount of results relating to   Hilbert's irreducibility theorem in the context of algebraic varieties, we are not able to cover all of them.
Rather, we try to structure and compare some of the important results,
while still aiming to at least mention every paper that is directly related to the Hilbert property of varieties in the sense of Serre, or generalizations thereof.
As most of the theorems we state are a summary of the work of many people, we do not attribute theorems by names (with very few exceptions, like Hilbert's irreducibility theorem) but rather try to explain the various contributions in the ``proof'' following the theorem.
So apart from precise references, these proofs also contain some historical information, and in some cases hints to stronger and/or more general statements.

\subsection*{Acknowledgements} 
The authors are grateful to
Finn Bartsch,
Lior Bary-Soroker,
Pietro Corvaja, 
Pierre D\`ebes,
Julian Demeio, 
Davide Lombardo,
Dan Loughran,
Cedric Luger,
Sebastian Petersen,
Sam Streeter,
and Umberto Zannier 
for the exchange of many ideas relevant to this survey.
They are indebted to Jean-Louis Colliot-Th\'el\`ene, Sebastian Petersen, Sam Streeter, Olivier Wittenberg and Umberto Zannier for numerous helpful remarks and corrections on a preliminary version.

\section{Hilbert irreducibility} 
\label{sec:one}

\subsection{Hilbert's theorem}
\label{sec:HIT}

Hilbert's irreducibility theorem is the following result proven in \cite[Theorem~I]{Hilbert}:

\begin{theorem}\label{thm:Hilbert}
For every irreducible $f\in \mathbb{Q}[t,x]\setminus\mathbb{Q}[t]$ there exist infinitely many
$\tau\in\mathbb{Q}$ with $f(\tau,x)$ irreducible in $\mathbb{Q}[x]$.
\end{theorem}

\begin{remark}\label{rem:Hilbert}
In fact, Hilbert obtained infinitely many $\tau\in\mathbb{Z}$ with $f(\tau,x)$ irreducible.
He also deduced several generalizations:
over number fields instead of $\mathbb{Q}$,
for several polynomials instead of one $f$, for polynomials in several variables $t_1,\dots,t_n$ instead of $t$,
and in several variables $x_1,\dots,x_m$ instead of $x$.
We will state some of these variants explicitly later on
(see in particular Theorem~\ref{thm:Hilbertian} and Proposition~\ref{prop:Hilbertian}).
\end{remark}

\begin{example}
If $g\in\mathbb{Z}[t]$ is such that $g(\tau)$ is a perfect square for all sufficiently large $\tau\in\mathbb{Z}$, then $g$ is itself a square of another polynomial:
Otherwise one could apply Theorem \ref{thm:Hilbert} in the version for $\tau\in\mathbb{Z}$ (Remark \ref{rem:Hilbert}) to $f(t,x)=x^2-g(t)$.
\end{example}

Hilbert's motivation came from Galois theory:
Using the version of his theorem for several variables $t_1,\dots,t_n$ instead of $t$,
he deduces the following corollary \cite[Theorem IV]{Hilbert} and example:

\begin{corollary}\label{cor:Hilbert}
Let $F=\mathbb{Q}(t_1,\dots,t_n)$ be the rational function field in $n$ variables over $\mathbb{Q}$.
For every $f\in F[x]$ with Galois group $G:={\rm Gal}(f/F)$, there exists $\underline{\tau}\in\mathbb{Q}^n$
such that 
$f(\underline{\tau},x)\in\mathbb{Q}[x]$
is defined
and
${\rm Gal}(f(\underline{\tau},x)/\mathbb{Q})\cong G$.
\end{corollary}

\begin{example}
Let 
$f=\prod_{i=1}^n(x-t_i)=x^n+\sum_{i=1}^n(-1)^is_ix^{n-i}$ 
where
$s_i$ is the $i$-th elementary symmetric polynomial in $t_1,\dots,t_n$.
Then $f\in \mathbb{Q}(s_1,\dots,s_n)[x]$ has splitting field $\mathbb{Q}(t_1,\dots,t_n)$
and Galois group (over $F=\mathbb{Q}(s_1,\dots,s_n)$) the symmetric group $S_n$.
Indeed, $F$ is precisely the fixed field of 
the action of $S_n$ on
$\mathbb{Q}(t_1,\dots,t_n)$ 
by permuting the $t_i$.
It follows also that $s_1,\dots,s_n$ are algebraically independent,
and so we can apply Corollary \ref{cor:Hilbert} to $f$ to obtain
polynomials over $\mathbb{Q}$ with Galois group $S_n$.
\end{example}

In geometric terms, this argument of Hilbert exploits that the quotient variety $\mathbb{A}^n_\mathbb{Q}/S_n$ is again a rational variety (i.e., birationally equivalent to $\mathbb{A}^n_\mathbb{Q}$), in fact $\mathbb{A}^n_\mathbb{Q}/S_n\cong\mathbb{A}^n_\mathbb{Q}$, so that the Galois cover $\mathbb{A}^n_\mathbb{Q}\rightarrow\mathbb{A}^n_\mathbb{Q}/S_n$ can be specialized to a Galois extension of $\mathbb{Q}$ with the same Galois group.
(We will return to this geometric interpretation in Section \ref{sec:HP}.)
In what became known as the {\em Noether program},
Emmy Noether \cite{Noether} started the investigation of whether 
$\mathbb{A}^n_\mathbb{Q}/G$ is rational for all finite groups $G\leq S_n$, 
which if true would give a strategy to solve the inverse Galois problem:
Every finite group $G$ would be the Galois group of a Galois extension of $\mathbb{Q}$.
However, counterexamples found by Swan \cite{Swan} and Voskresenski\u{\i} \cite{Voskresenskii} 
show that $\mathbb{A}^n_\mathbb{Q}/G$ is non-rational for  $G=\mathbb{Z}/47\mathbb{Z}$,
and later the same was shown for    $G=\mathbb{Z}/8\mathbb{Z}$, see the references in \cite{Lenstra} and \cite[\S 1.6.1]{WittenbergPark} for more on the history of this problem.

\begin{remark}
Apart from Galois theory, Hilbert's theorem has found numerous applications in number theory and diophantine geometry.
For example, N\'eron \cite{Neron} used it to specialize abelian varieties from $\mathbb{Q}(t_1,\dots,t_n)$ to $\mathbb{Q}$, which was used to construct elliptic curves over $\mathbb{Q}$ of high rank,
see e.g.~\cite[Chapter 11]{Serre_MW}.
See also \cite{Silverman_specialization} for a strengthening in the case $n=1$,
and \cite{CT,LS,CS} for recent advances. 

Another example application
is the question
asked by H.~Friedman 
whether there exists a polynomial $f(x,y)\in \mathbb{Q}[x,y]$ such that the associated map $\tilde{f}\colon\mathbb{Q}\times \mathbb{Q}\to \mathbb{Q}$ is injective. 
Although this question remains currently unsolved, \cite[Theorem~1.1]{Poo10} uses Hilbert's irreducibility theorem to prove that such a polynomial exists if
the set of rational points of
a certain specific irreducible hypersurface in $\mathbb{P}^3_\mathbb{Q}$ 
is not Zariski-dense,
and \cite{Bre21} 
gives a proof,
conditional on a weak version of Lang's conjecture (see Conjecture \ref{conjecture:lang_mordellic} below),
of a   higher-dimensional generalization of Hilbert's irreducibility theorem first explored systematically in \cite[Section 3.8.1]{CZbook} (there called a `Hilbert property for fibrations'), which would imply that $\tilde{f}$ is never {\em bijective}.
\end{remark}

\begin{remark}
Hilbert's proof of Theorem \ref{thm:Hilbert} works with Puiseux series expansions (in $t$) of the roots of $f(x)$. See also \cite{VGR} for a modern account of Hilbert's approach and the combinatorics involved.
Hilbert's proof was simplified by Dörge \cite{Dörge}, see for example \cite{Völklein} for a modern exposition.
A non-standard proof was given by Roquette \cite{Roquette} following an approach to Hilbertian fields (see Section \ref{sec:Hilbertian}) of Gilmore and Robinson.
Most other proofs use reductions modulo different prime numbers, like 
Fried's proof \cite{Fried} that uses the Riemann hypothesis for curves,
or the one that follows from Theorem \ref{thm:WWA} below that uses the Lang--Weil estimates.

Several strengthenings exist:
For example, 
\cite{Schinzel} obtains a whole arithmetic progression $a+b\mathbb{Z}$
of $\tau$ for which $f(\tau,x)$ is irreducible,
and also Fried's proof mentioned above gives this.
Quantitative versions of Theorem \ref{thm:Hilbert}
give estimates for the number of $\tau$ in $\mathbb{Z}$ or $\mathbb{Q}$ such that $f(\tau,x)$ is irreducible, like \cite[Theorem 2.5]{Cohen} that uses the large sieve inequality, phrased in \cite[Theorem 3.4.4]{Serre} in terms of thin sets (cf.~Section \ref{sec:HP}).
For more recent developments see for example \cite{Zywina},
and see \cite{Zan98} for a version uniform across number fields. 
Effective versions, 
starting with \cite{Debes,SchinzelZannier},
give upper bounds on the smallest $\tau\in\mathbb{N}$ with $f(\tau,x)$ irreducible,
see \cite{PS,CDHNV} for recent results.
\end{remark}

\subsection{Hilbertian fields}
\label{sec:Hilbertian}

Lang \cite{Lang_diophantine_geometry} coined the term {\em Hilbertian} for fields $K$
over which the statement of Hilbert's irreducibility theorem holds:

\begin{definition}\label{def:Hilbertian}
A field $K$ is {\em Hilbertian}    
if for every irreducible $f\in K[t,x]$, separable and of positive degree in $x$,
there exist infinitely many $\tau \in K$ with $f(\tau,x)\in K[x]$ irreducible.
\end{definition}

\begin{example}
Finite fields and algebraically closed fields are trivially not Hilbertian.
The same holds for $K=\mathbb{R}$ (take e.g.~$f=x^2-t^2-1$) and $K=\mathbb{Q}_p$ (cf.~\cite[Example 15.5.5]{FJ}).
In fact, no local field is Hilbertian.
\end{example}

\begin{theorem}\label{thm:Hilbertian}
The field $K$ is Hilbertian in each of the following cases:
\begin{enumerate}
    \item $K$ is a number field.
    \item $K$ is finitely generated and transcendental over some field $k$.
\end{enumerate}
\end{theorem}

\begin{proof}
(1) was in fact already proven by Hilbert \cite{Hilbert}.
(2) is a Bertini-type argument, see e.g.~\cite[Theorem 13.4.2]{FJ}.
\end{proof}

In particular, all global fields 
and more generally all infinite fields that are {\em finitely generated} (i.e.~finitely generated over their prime field) are Hilbertian.
Many more fields are known to be Hilbertian, 
in particular various fields of power series, like
${\rm Frac}(\mathbb{C}[\![x,y]\!])$ and ${\rm Frac}(\mathbb{Z}[\![t]\!])$, cf.~\cite{Weissauer,FehmParan}.
In Section~\ref{sec:basechange} we will discuss results regarding when 
an algebraic extension of a Hilbertian field is again Hilbertian.
Definition~\ref{def:Hilbertian} is equivalent to several variants of it:

\begin{proposition}\label{prop:Hilbertian}
For a field $K$, the following are equivalent:
\begin{enumerate}
    \item $K$ is Hilbertian.
    \item For every 
    $r,n\in\mathbb{N}$ and every
    $f_1,\dots,f_r\in K[t_1,\dots,t_n,x]$ that are irreducible in $K(\underline{t})[x]$
    and separable in $x$,
    and for every $0\neq g\in K[\underline{t}]$
    there exists $\underline{\tau}\in K^n$ with $g(\underline{\tau})\neq 0$
    such that $f_1(\underline{\tau},x),\dots,f_r(\underline{\tau},x)\in K[x]$ are irreducible.
    \item For every $r,n\in\mathbb{N}$ and every absolutely irreducible $f_1,\dots,f_r\in K[t_1,\dots,t_n,x]$ that are monic, separable and of degree at least $2$ in $x$
    there exists $\underline{\tau}\in K^n$ such that each $f_i(\underline{\tau},x)$ has no zero in $K$.
    \item For every absolutely irreducible $f\in K[t,x]$ that is separable and of positive degree in $x$, there exist infinitely many $\tau \in K$ with $f(\tau,x)\in K[x]$ irreducible.
\end{enumerate}
\end{proposition}

\begin{proof}
$(1)\Rightarrow(2)$ is \cite[Prop.~13.2.2]{FJ},
$(2)\Rightarrow(3)$ is \cite[Lemma 12.1.6]{FJ},
$(3)\Rightarrow(1)$ is \cite[Lemma 13.1.2]{FJ}, 
$(1)\Rightarrow(4)$ is trivial
and $(4)\Rightarrow(1)$ is \cite[Theorem 1.1]{BarySoroker}.    
\end{proof}

\begin{remark}\label{rem:PAC}
Hilbertian fields play an important role in Galois theory, model theory and field arithmetic.    
For example, it is conjectured that over every Hilbertian field $K$,
the inverse Galois problem has a positive answer,
and, more generally,
that every so-called finite split embedding problem is solvable \cite{DD}.
This was proven in \cite{Pop} in the case that $K$ is {\em large},
i.e.~every smooth $K$-curve $X$ satisfies $|X(K)|\in\{0,\infty\}$.
If a field $K$ is {\em pseudo-algebraically closed} (also called PAC),
i.e.~every geometrically integral $K$-variety $X$ satisfies $X(K)\neq\emptyset$,
a celebrated theorem of Fried-Völklein and Roquette states that $K$ is Hilbertian if and only if the absolute Galois group of $K$ is $\omega$-free \cite[Theorem 5.10.3]{Jarden_patching}.
This allows a satisfactory description of the first-order theory of Hilbertian PAC fields 
and makes them an object of intense study in model theory, see e.g.~\cite{Chatzidakis}.
\end{remark}

\begin{remark}\label{rem:Hilbertian_def}
Extensive sources on Hilbertian fields are
\cite[Chapter 9]{Lang_fundamentals},
\cite[Chapters 12-16]{FJ},
\cite[Chapter IV]{MM}, and
\cite[Chapter 1]{Völklein}.
For $K$ of positive characteristic, 
there are also variants of Definition \ref{def:Hilbertian} in the literature that are not equivalent to it.
Our definition (more precisely Proposition~\ref{prop:Hilbertian}(2))
is the one used in \cite{FJ} from the second edition onwards
and is the one commonly used nowadays.
There are also many actual strengthenings and weakenings in the literature, 
and we mention only some of them:
fully Hilbertian fields \cite{BP},
$g$-Hilbertian fields \cite[Chapter 13.6]{FJ},
RG-Hilbertian fields and R-Hilbertian fields \cite{DH},
real Hilbertian fields \cite{FHV},
$\Theta$-Hilbertian fields \cite{FH},
and Brauer-Hilbertian fields \cite{FSS}.
There are also the notions of  Hilbertian rings \cite[Chapter 13.4]{FJ}
and integrally Hilbertian rings \cite{BD}.
\end{remark}

\subsection{The Hilbert property}
\label{sec:HP}

Motivated by the failure of the Noether program,
Colliot-Th\'el\`ene and Sansuc \cite{CTS}
investigated whether it might be possible to specialize
the cover $\mathbb{A}_\mathbb{Q}^n\rightarrow\mathbb{A}_\mathbb{Q}^n/G$ to a 
Galois extension of $\mathbb{Q}$ with Galois group $G$
even if $\mathbb{A}^n_\mathbb{Q}/G$ is not rational.
This led to the following definition by Serre \cite{Serre} (see Remark \ref{rem:HP} below for historic details).
Here, $X$ is a $K$-variety (i.e.~an integral separated scheme of finite type over $K$),
and a {\em cover} is a finite surjective generically \'etale morphism of $K$-varieties.

\begin{definition}\label{def:thin}\label{def:HP}
A subset $\Sigma\subseteq X(K)$ is \emph{thin (in $X$)} if there are $n\in\mathbb{N}$
and for each $i\in\{1,\dots,n\}$ a $K$-variety $Y_i$ and a cover $\pi_i\colon Y_i\to X$
with ${\rm deg}(\pi_i)>1$ such that
$\Sigma \setminus \bigcup_{i=1}^n \pi_i(Y_i(K))$
is not Zariski-dense in $X$. 
The variety $X$ has the {\em Hilbert property (``$X$ has HP'')} if $X(K)$ is not thin in $X$.
\end{definition}

\begin{example}
By definition, if $X(K)$ is not Zariski-dense in $X$, then $X$ does not have HP.
So for example, by Faltings's theorem, no curve $X$ of genus $g_X>1$ over a number field $K$ has HP.
\end{example}

\begin{example}\label{ex:abelian_var_not_HP}
An elliptic curve $E$ over $\mathbb{Q}$ does not have HP:
By Mordell's theorem, $2E(\mathbb{Q})$ has finite index in $E(\mathbb{Q})$.
Thus if $P_1,\dots,P_r \in E(\mathbb{Q})$ is a system of coset representatives of $2E(\mathbb{Q})$, then $E(\mathbb{Q})=\bigcup_{i=1}^r\pi_i(E(\mathbb{Q}))$,
where $\pi_i\colon E\rightarrow E$ is the cover of degree $4$ given by
$P\mapsto 2P+P_i$.
The same argument works for any nonzero abelian variety $A$
over a finitely generated field $K$ of characteristic zero,
as $A(K)$ is finitely generated by the theorems of Mordell--Weil and Lang--N\'eron, cf.~\cite[Exercise 4 on p.~20]{Serre}.
\end{example}

\begin{remark}\label{rem:HP}\label{rem:HP_birat}
By definition, HP is a birational property,
so if $X$ is birationally equivalent to a $K$-variety $X'$ (written $X\sim X'$),
then $X$ has HP if and only if $X'$ does.
In Definition \ref{def:HP} one can replace the finite surjective generically \'etale morphisms $\pi\colon Y_i\rightarrow X$ by generically finite generically \'etale dominant rational maps $\pi_i\colon Y_i\dashrightarrow X$ without changing the class of thin sets.
Similarly, one could demand that the $Y_i$ are geometrically integral, since otherwise $\pi_i(Y_i(K))$ is not Zariski-dense in $X$.
\end{remark}

\begin{example}\label{ex:HP_implies_Hilbertian}
$\mathbb{A}^n_K$ has HP if and only if $K$ is Hilbertian, as follows essentially from Proposition~\ref{prop:Hilbertian}(3) and Remark \ref{rem:HP}.
For example, 
if $K$ is finitely generated and infinite,
then every $K$-variety $X$ 
that is rational (i.e.~$X\sim\mathbb{A}_K^n$ for some $n$) has HP (Theorem \ref{thm:Hilbertian} and Remark \ref{rem:HP}).
\end{example}

\begin{example}\label{ex:PAC}
If $K$ is Hilbertian and PAC (cf.~Remark \ref{rem:PAC}),
then every $K$-variety has HP (cf.~\cite[Prop.~5.1]{BFP24}).
An example of such a field is $\mathbb{Q}^{\rm tr}(\sqrt{-1})$,
where $\mathbb{Q}^{\rm tr}$ is the field
of totally real algebraic numbers,
which itself is neither Hilbertian nor PAC,
cf.~\cite[Prop.~5.6]{BFP24}.
\end{example}

The connection to the Noether program is given by the following observation, 
which is the geometric version of the equivalence of (2) and (3) in Proposition~\ref{prop:Hilbertian}:

\begin{proposition}
A $K$-variety $X$ has HP if and only if
for every $K$-variety $Y$ and every Galois cover $\pi\colon Y\rightarrow X$ 
there exists a Zariski-dense set of $x\in X(K)$ such that the fiber $\pi^{-1}(x)$ is integral.
\end{proposition}

\begin{proof}
For $\Rightarrow$, 
if $\pi\colon Y\rightarrow X$ is a Galois cover with Galois group $G$,
$\pi^{-1}(x)$ is integral for every $x\in X(K)$ over which $\pi$ is unramified
and which does not lie in the image of $(Y/H)(K)$ for the finitely many proper subgroups $H$ of $G$, cf.~\cite[Prop.~3.3.1]{Serre}.
For $\Leftarrow$, if $\pi_i\colon Y_i\rightarrow X$ are covers with ${\rm deg}(\pi_i)>1$,
let $Y$ be the normalization of $X$ in the Galois closure of the compositum of the function fields
$K(Y_i)$ (embedded in an algebraic closure of $K(X)$).
Then every $x\in X(K)$ with $\pi^{-1}(x)$ integral does not lie in $\bigcup_{i=1}^n\pi_i(Y_i(K))$.  
\end{proof}

So if $X=\mathbb{A}^n_\mathbb{Q}/G$ has HP, then there exists $x\in X(K)$ for which $\pi^{-1}(x)$ is integral, where $\pi\colon \mathbb{A}^n_\mathbb{Q}\rightarrow X$ is as above, and then
$\pi^{-1}(x)={\rm Spec}(L)$ for a Galois extension $L/K$ with ${\rm Gal}(L/K)\cong G$,
similar to Corollary \ref{cor:Hilbert}.
This raised the questions whether all unirational varieties over $\mathbb{Q}$
(i.e.~all $X$ which admit a dominant rational map $\mathbb{A}^n_\mathbb{Q}\dashrightarrow X$ for some $n$)
have HP,
and which varieties over number fields in general have HP.
Regarding the first question, conjecturally
all unirational varieties over number fields have HP:

\begin{definition}\label{def:WWA}
A variety $X$ over a number field $K$  satisfies {\em weak weak approximation (``$X$ satisfies WWA'')} if there is a finite set $S_0$ of places of $K$ such that
for every finite set $S_1$ of places of $K$ with $S_1\cap S_0=\emptyset$,
the image of the diagonal embedding $X(K)\rightarrow\prod_{v\in S_1}X(K_v)$
is dense, where $K_v$ denotes the completion of $K$ at $v$.
\end{definition}

For example, $\mathbb{A}^n_K$ satisfies WWA (in fact with $S_0=\emptyset$)
by the usual Artin--Whaples weak approximation theorem \cite[Theorem XII.1.2]{Lang_algebra}, and therefore so does every smooth rational $K$-variety, cf.~\cite[Lemma 3.5.5]{Serre}.

\begin{theorem}\label{thm:WWA}
If a variety $X$ over a number field $K$ satisfies WWA, then $X$ has HP.
\end{theorem}

\begin{proof}
This is \cite[Theorem 3.5.7]{Serre}, where it is 
deduced from the Lang--Weil estimates
and is attributed to \cite{Ekedahl} and 
a letter of Colliot-Th\'el\`ene to Ekedahl from 1987\footnote{mistakenly dated 1988 in \cite{Serre}}. This letter is also the reference given for the following conjecture.
\end{proof}

\begin{conjecture}\label{conj:WWA}
Every smooth unirational variety over a number field satisfies WWA, so in particular it has HP.
\end{conjecture}

Regarding the second question, Corvaja and Zannier observed an important necessary condition,
and they  raised the question whether this might possibly be the only one:
 
\begin{theorem}\label{thm:CZ}
Let $X$ be a normal projective variety over a finitely generated field $K$.
If $X$ has HP, then $X_{\overline{K}}$ is simply connected in the sense that every   \'etale cover $Y\rightarrow X_{\overline{K}}$ is trivial.  
\end{theorem}

\begin{proof}
In the case where $K$ is a number field,
this is \cite[Theorem 1.4]{CZ}, which proves and uses a variant of the Chevalley--Weil theorem for number fields.
The same proof goes through for finitely generated $K$
by using the corresponding generalization of the Chevalley--Weil theorem in \cite[Theorem 3.8]{CDJLZ}; see \cite[Theorem~3.2]{Luger}.
\end{proof}

\begin{question}[{\cite[Question-Conjecture 1]{CZ}}]\label{conj:CZ1}
Does every
smooth projective variety $X$ over a number field $K$
with $X_{\overline{K}}$ simply connected
and $X(K)$ Zariski-dense in $X$ have HP?
\end{question}

\begin{remark}\label{rem:smooth_necessary}
As evidence that the answer to Question \ref{conj:CZ1} might be positive, Corvaja and Zannier draw conclusions about Kummer varieties \cite[p.~589]{CZ} which they then prove in the case of Kummer surfaces \cite[Appendix]{CZ},
also cf.~Section \ref{sec:kappa0}.
However, they also give an example \cite[Theorem 1.5]{CZ} that shows that the smoothness assumption is necessary,
and they suggest that it might be necessary to
allow for a finite extension of the field $K$ for the answer to be positive. 
\end{remark}

Theorem \ref{thm:CZ} gives another explanation for why
abelian varieties over number fields do not have HP (Example \ref{ex:abelian_var_not_HP}).

\begin{remark}\label{rem:HP}
The Hilbert property and thin sets are discussed in the textbooks
\cite[Chapter 3]{Serre},
\cite[Chapter 9]{Serre_MW}, and
\cite[Chapter 13.5]{FJ}.
The first definition of HP in \cite[p.~189]{CTS} (there called ``Hilbert type'') was phrased in terms of specializations of polynomials,
\cite{Serre} introduces the notion of thin sets,
while the definition we give here is adjusted for positive characteristic (cf.~Remark \ref{rem:Hilbertian_def}).
Some results on HP predate its formulation as a property, like \cite[Theorem 46.1]{Manin}, whose first Russian edition appeared in 1972.
Thin sets show up in many other places,
for example as exceptional sets in modern formulations of Manin's conjecture on the asymptotic growth of rational points of bounded height on Fano varieties, see the references in \cite{BL,LST}.
\end{remark}

\begin{remark}
Conjecture \ref{conj:WWA} is proven in several cases. 
Generalizations and strengthenings of Conjecture \ref{conj:WWA} exist:
For example for the bigger class of rationally connected varieties,
and with a concrete finite set $S_0$ in Definition \ref{def:WWA}, coming from the Brauer--Manin obstruction \cite{CT_points_rationnels}.
The latter strengthening was proven for example over number fields for $X={\rm SL}_n/G$ with $G\leq{\rm SL}_n(K)$ a  supersolvable finite group \cite[Th\'eor\`eme B]{HW2},
which allows to draw Galois theoretic conclusions for these groups beyond the inverse Galois problem, e.g.~so-called Grunwald problems \cite[Corollaire p.~777]{HW2}.
\end{remark}

\begin{remark} 
Another topic that we do not cover in this survey, but which is likely to become increasingly important, is the Hilbert property for algebraic stacks. Recent work has shown that the correct formulation of several arithmetic conjectures requires one to understand thin sets on stacks and not only on varieties; see for example  \cite[Section~9]{DardaYasuda} and \cite[Sections~2.2, 3.7, 10.3 and 10.4]{LoughranSantens}. 
\end{remark}

\subsection{The weak Hilbert property}
Their observation (Theorem \ref{thm:CZ}) motivated Corvaja and Zannier to introduce the following weakening of the Hilbert property in \cite[Section 2.2]{CZ}.
In the following, 
let $K$ always be a field of characteristic zero
and $X$ a normal $K$-variety.
A cover is {\em ramified} if it is not \'etale.

\begin{definition} 
A subset $\Sigma\subseteq X(K)$ is \emph{strongly thin (in $X$)} if there are $n\in\mathbb{N}$
and for each $i\in\{1,\dots,n\}$ a normal $K$-variety $Y_i$ and a ramified cover $\pi_i\colon Y_i\to X$
such that $\Sigma \setminus \bigcup_{i=1}^n \pi_i(Y_i(K))$
is not Zariski-dense in $X$. 
The variety $X$ has the {\em weak Hilbert property (``$X$ has WHP'')} if $X(K)$ is not strongly thin in $X$. 
\end{definition}

Trivially, if $X$ has HP, then $X$ has WHP,
and if $X_{\overline{K}}$ is simply connected (in the sense of Theorem~\ref{thm:CZ}), then the converse holds.
Again, if $X(K)$ is not Zariski-dense in $X$, then $X$ does not have WHP.

\begin{example}\label{ex:E_WHP}
If $K$ is a number field and $X=E$ is an elliptic curve of positive rank, then $E$ has WHP:
If $\pi_i\colon Y_i\rightarrow E$ is a ramified cover, then $Y_i$ has genus at least $2$ by the Riemann--Hurwitz formula,
and therefore $Y_i(K)$ is finite by Faltings's theorem.
In fact, this argument shows that no infinite $\Sigma\subseteq X(K)$ is strongly thin.
See also Section \ref{sec:curves} where we discuss curves in more detail,
and Section \ref{sec:AV} for generalizations to abelian varieties.
\end{example}

We have the following analogue of Remark \ref{rem:HP_birat}:

\begin{proposition}\label{prop:WHP_birat}
If $X$ and $X'$ are smooth proper geometrically connected $K$-varieties with $X\sim X'$,
then $X$ has WHP if and only if $X'$ does.
\end{proposition}

\begin{proof}
This is \cite[Prop.~3.1]{CDJLZ}.  
\end{proof}

\begin{remark}\label{rem:WHP_birat}
While it is not known whether the ``proper'' can be dropped from Proposition~\ref{prop:WHP_birat},
an example in \cite{CZ} over $K=\mathbb{Q}$ (cf.~Remark \ref{rem:smooth_necessary}) shows that ``smooth'' cannot be weakened to ``normal'',
cf.~\cite[Remark 7.11]{BJL}.
\end{remark}

If $\mathbb{A}^1_K$ has WHP, then it has HP, and so $K$ is Hilbertian (cf.~Example \ref{ex:HP_implies_Hilbertian}).
In fact, varieties with WHP exist only over Hilbertian fields:

\begin{proposition}\label{prop:WHP_implies_Hilbertian}
The following are equivalent:
\begin{enumerate}
    \item $K$ is Hilbertian.
    \item Every rational $K$-variety has HP.
    \item Some normal $K$-variety of positive dimension has WHP.  
\end{enumerate}
\end{proposition}

\begin{proof}
See \cite[Proposition 5.5]{BFP24}.
The weaker statement with HP instead of WHP 
appears already in \cite[p.~20, Exercise 1]{Serre}.
\end{proof}

In fact, unramified covers were the only serious obstruction to the Hilbert property that Corvaja and Zannier saw: 

\begin{question}[{\cite[Question-Conjecture 2]{CZ}}]\label{conj:CZ2}
Does every smooth projective variety $X$ over a number field $K$
with  $X(K)$ Zariski-dense in $X$
have WHP?
\end{question}

\begin{remark}\label{rem:CZ2}
In fact, \cite{CZ} ask the question with $X$ only required to be normal, but then the answer is negative,
as Remark \ref{rem:WHP_birat} shows.
As evidence that the answer to Question \ref{conj:CZ2} might be positive, Corvaja and Zannier mention the case of curves
(where it holds, see Example~\ref{ex:E_WHP} and Section~\ref{sec:curves}),
conjectures of Vojta, and Manin's conjecture in the case of Fano varieties.
Like with Question \ref{conj:CZ1}, they suggest that the answer might be positive only after allowing a finite extension of the base field.
We discuss this in more detail in Section \ref{sec:denseness}.
\end{remark}

\begin{remark}
The weak Hilbert property was first defined in \cite{CZ}, 
but at least special cases  occurred implicitly already in 
\cite{FZ,CorvajaAlgebraicGroups,Zannier}. 
The term ``strongly thin'' was introduced later in \cite{BFP24,Luger2}.
Corvaja and Zannier \cite[p.~582]{CZ} make the claim that WHP
``would admit in principle much the same applications as the HP'',
whose validity has yet to be explored. 
\end{remark}

\subsection{Integral Hilbert properties}
\label{sec:integral}
There are several variants of HP and WHP in the literature, both explicit and implicit, that refer to integral rather than rational points.
We will mention such results only occasionally,
and the reader interested mainly in HP and WHP may safely skip this subsection.
To talk about integral points, we need to work with models of our varieties over $\mathbb{Z}$ or whatever the ring of ``integers'' of $K$ should be:

Let $R$ be an integral domain with fraction field $K$ and let $X$ be a quasi-projective variety over $K$. A \emph{model for $X$ over $R$} is a quasi-projective flat integral $R$-scheme $\mathcal{X}$ together with an isomorphism $\mathcal{X}_K\cong X$ over $K$.  
We will usually omit the isomorphism from our notation. 
If $\mathcal{X}$ is a model for $X$ over $R$, then $\mathcal{X}(R)$ is naturally a subset of $X(K)$.

\begin{example} 
Assume $X=\mathrm{Spec}( K[x_1,\ldots,x_n]/(f_1,\ldots,f_r))\subseteq \mathbb{A}^n_K$ is the zero locus of a finite collection of polynomials   $f_1,\ldots, f_r\in R[x_1,\dots,x_n]$ with $(f_1,\ldots,f_r)$ prime in $R[x_1,\ldots,x_n]$.  
Then  $\mathcal{X} = \mathrm{Spec}(R[x_1,\ldots,x_n]/( f_1,\ldots,  f_r))$ is a model for $X$ over $R$. 
\end{example}

If $K$ is a number field and $S$ a finite set of finite places of $K$,
we denote by $\mathcal{O}_{K,S}$ the ring of $S$-integers of $K$.
In this setting,
if $X$ is a variety over $K$
and $\mathcal{X}$ is  a  fixed model of $X$ over $\mathcal{O}_{K,S}$,
\cite{Coccia}
defines $X$  to have the {\em (weak) $S$-integral Hilbert property }
(IHP/WIHP) if $\mathcal{X}(\mathcal{O}_{K,S})$ is not (strongly) thin in $X$.
A more general definition can be found in \cite{Luger},
which would in this situation say that $\mathcal{X}$ has the (weak) Hilbert property over $\mathcal{O}_{K,S}$.  

\begin{remark} 
If $R$ is a regular noetherian integral domain with fraction field $K$ and $\mathcal{X}$ is a model for $X$ over $R$,
\cite{Luger} defines $\mathcal{X}$ to have the Hilbert property over $R$
if the set of near-$R$-integral points on $\mathcal{X}$ is not thin in $X(K)$,
where a $K$-point $P\in X(K)$ is {\em near-$R$-integral} (as defined by Vojta \cite[\S4]{Vojta}) if there is an open subset $U\subseteq \mathrm{Spec}(R)$ with complement of codimension at least two and a morphism $U\to \mathcal{X}$ extending $P$. 
However, if $\dim R = 1$ (e.g., $R= \mathcal{O}_{K,S}$ for $K$ a number field), then every near-$R$-integral point is an $R$-point,
and we will not discuss results on near-integral points in this survey. 
\end{remark}

\begin{example}
Let $X=\mathbb{A}^1_\mathbb{Q}$ and $X'=\mathbb{A}^1_\mathbb{Q}\setminus\{0\}$,
with models
$\mathcal{X}={\rm Spec}(\mathbb{Z}[t])$ respectively
$\mathcal{X}'={\rm Spec}(\mathbb{Z}[t,t^{-1}])$ over $\mathbb{Z}$.
By Hilbert's irreducibility theorem (Remark \ref{rem:Hilbert}),  $X(\mathbb{Z})$ is not thin in $X$,
but
 $X'(\mathbb{Z})=\{1,-1\}$ is thin in $X'$ (equivalently, in $X$).
Similarly,
if $K$ is a number field and $S$ a finite set of finite places of $K$, then $X'(\mathcal{O}_{K,S})=\mathcal{O}_{K,S}^\times$ is thin in $X'_K$ and strongly thin in $X_K$,
by Dirichlet's unit theorem (cf.~the argument in Example \ref{ex:abelian_var_not_HP}).
For example,
every element of
 $\mathcal{O}_{\mathbb{Q},\{2\}}^\times=\mathbb{Z}[\frac{1}{2}]^\times\subseteq\mathbb{Q}^\times$
is in the image of $X'(\mathbb{Q})$ under one of 
the four ramified covers $X\to X$ given by $z\mapsto z^2$, $z\mapsto -z^2$, $z\mapsto 2z^2$, and $z\mapsto -2z^2$.
However, $X'(\mathcal{O}_{K,S})$ is not always strongly thin in $X'_K$:  
\end{example}

\begin{proposition}\label{thm:G_m}
    Let $K$ be a number field and let $S$ be a finite set of finite places of $K$. If  $\mathcal{O}_{K,S}^\times$ is infinite, then   $\mathcal{O}_{K,S}^\times$ is not strongly thin in  $\mathbb{A}^1_{K}\setminus \{0\}$. 
\end{proposition}
\begin{proof}
    This is similar to the proof of Example \ref{ex:E_WHP}:
    If $\pi\colon C\to \mathbb{A}^1_K\setminus \{0\}$ is a  ramified cover with $C$ a smooth curve,
    let $\overline{C}$ be a smooth projective curve containing $C$ of genus $g_{\overline{C}}$, and let $\overline{\pi}\colon \overline{C}\to \mathbb{P}^1_K$ be the unique extension of $\pi$. 
    Applying the Riemann--Hurwitz formula to $\overline{\pi}$, we see that $2g_{\overline{C}} -2 + \#(\overline{C}_{\overline{K}}\setminus C_{\overline{K}}) >0$,
 so Siegel's theorem (more precisely Mahler's generalization thereof, cf.~\cite[Theorem 1]{CZ02}), gives that  any model of $C$ over $\mathcal{O}_{K,S}$ has only finitely many $\mathcal{O}_{K,S}$-points. 
 This shows that 
 if $\mathcal{O}_{K,S}^\times$ is infinite, then it  is not strongly thin. 
 In fact, $\mathcal{O}_{K,S}^{\times}$ contains an infinite cyclic subgroup $\Gamma$, 
 and already $\Gamma$ is not strongly thin in $\mathbb{A}^1_K\setminus\{0\}$
 by
 \cite[Theorem~4]{Zannier}
 (cf.\ Theorem~  \ref{thm:alg_groups_fingen} below),
 which does not use Siegel's theorem. 
\end{proof}

\subsection[Potential density of rational points]{Potential density\footnote{or {\em denseness}, emphasizing that this is not meant in a quantitative sense} of rational points}
\label{sec:denseness}
Let $K$ be a field of characteristic zero and $X$ a normal and geometrically integral $K$-variety.
A positive answer to Question \ref{conj:CZ2} would mean that
if $K$ is a number field, $X$ is smooth and projective, and $X(K)$ is Zariski-dense in $X$, then $X$ has WHP, but Corvaja and Zannier considered the possibility (cf.~Remark~\ref{rem:CZ2}) that this might hold only potentially, in the following sense: 

\begin{definition}
The $K$-variety $X$  has
{\em potential HP}, respectively {\em potential WHP},
if $X_L$ has HP, respectively WHP, for some finite field extension $L/K$. 
\end{definition}

So in the case where $K$ is a number field,
a positive answer to this weaker version of Question \ref{conj:CZ2} would offer a precise prediction when $X$ should have potential WHP,
namely if and only if $X(L)$ is Zariski-dense in $X$ for some finite extension $L/K$.
We now discuss for which $X$ the latter condition is expected to be satisfied.
Results and conjectures in this subsection will only be used to motivate some questions and results later on.

Recall that a projective variety $X$ over $K$ is of \emph{general type} if there is a resolution of singularities $Y\to X$ such that the  
canonical divisor $K_Y$ 
is big. 
Equivalently, $X$ is of maximal Kodaira dimension $\kappa(X)={\rm dim}(X)$. 
Lang's conjecture \cite[Conjecture~5.7]{Lang_conjectures} predicts that varieties of general type over number fields have few rational points,
and therefore do not have HP:
\begin{conjecture}\label{conjecture:lang_mordellic}
A projective variety $X$ over a number field $K$ is of general type
if and only if there is a proper closed subset $\Delta\subsetneq X_{\overline{K}}$ such that $X(L)\setminus \Delta$ is finite for every finite field extension $L/K$.
\end{conjecture}

 In \cite[Definition~2.1]{Ca04} (see also \cite[Definition~8.1]{Ca11} or \cite[Definition~1.2]{BJL}) Campana introduced the class of {\em special} varieties 
and he conjectured that the arithmetic properties of 
 special varieties are  opposite to those conjectured by Lang for varieties of general type; see \cite[Conjecture~13.23]{Ca11}.  Combined with the expected positive answer to the weakened version (see Remark \ref{rem:CZ2}) of 
Question~\ref{conj:CZ2},
extended to finitely generated fields,
this gives:

\begin{conjecture}\label{conj:campana_corvaja_zannier}
If $K$ is finitely generated and $X$ smooth projective, 
the following are equivalent.
\begin{enumerate}
    \item $X$ has potential WHP.
    \item There is a finite field extension $L/K$ such that $X(L)$ is Zariski-dense in $X$.
    \item $X$ is special.
\end{enumerate}
\end{conjecture}

We will not make use of the precise definition, but for completeness we include a definition for projective varieties: If $X$ is projective over a field $K$ of characteristic zero, it is  
special if and only if there is a proper birational surjective morphism $Y \to X$ with $Y$ a smooth projective variety over $K$ with no Bogomolov sheaves,
where a line bundle $\mathcal{L}$ on $Y$ is a \emph{Bogomolov sheaf (for $Y$)} if there is an integer $1 \leq p \leq \dim Y$ such that there is a nonzero morphism $\mathcal{L} \to \Lambda^p \Omega^1_Y$ and the Iitaka dimension $\kappa(\mathcal{L})$ of $\mathcal{L}$,
 which is always at most $p$ \cite[\S 12, Theorem 4]{BogomolovHol},
 is equal to $p$.
This notion is
independent of the base field:  $X$ is special if and only if $X_L$ is special, where $L/K$ is any field extension \cite{BartschJavanpeykar}. 

\begin{example}\label{ex:special}
Examples of special varieties 
include algebraic groups \cite[Lemma~2.11]{Bartsch}, smooth projective (geometrically) rationally connected varieties (e.g., Fano varieties, cf.~Section~\ref{sec:Fano}) \cite[Corollary~2.28]{Ca04},  varieties with Kodaira dimension zero (e.g., hyperk\"ahler varieties and Calabi--Yau varieties) \cite[Theorem~5.1]{Ca04}. In particular, curves of genus at most one are special, but also  K3 surfaces, Enriques surfaces, and abelian varieties are special (as they have Kodaira dimension zero). Finally,  smooth hypersurfaces of degree at most $n+1$ in $\mathbb{P}^n$ are either Fano varieties or have Kodaira dimension zero, and are thus special. 
\end{example}

Campana established several properties of his class of special varieties.    
We record some of  his results here.

\begin{proposition}\label{thm:special} 
    Let $X$ and $Y$ be projective geometrically integral $K$-varieties.
    \begin{enumerate}
        \item If $X$ is special and $Y$ is birationally equivalent to $X$, then $Y$ is special.
        \item  If $X\to Y$ is a surjective morphism and $X$ is special, then $Y$ is special.
        \item   If $X\to Y$ is an \'etale cover, then $X$ is special if and only if $Y$ is special.
        \item If $X$ and $Y$ are special, then $X\times Y$ is special.
    \end{enumerate}
\end{proposition}
\begin{proof} 
For (1), choose  proper birational surjective morphisms $Z\to X$ and $Z\to Y$ with $Z$ a projective variety. Then $X$ is special if and only if $Z$ is special if and only if $Y$ is special by \cite[Lemma~2.8(2)]{BJL}. For (2) and (3) see \cite[Lemma~2.8(1,3)]{BJL}.  Finally, $(4)$ is well-known and follows from the more general \cite[Lemma~2.10]{Bartsch}. 
\end{proof}

By Conjecture \ref{conj:campana_corvaja_zannier}, the property ``$X$ has potential WHP'' should  behave similarly to the property    ``$X$ is special'' over finitely generated fields of characteristic zero. In particular, they should have the same formal properties (e.g., being closed under products),
and we record in Section \ref{sec:preservation} what is known in this direction. 
We also expect that some of these properties persist over (arbitrary!) Hilbertian fields, and therefore ask the following question about special varieties over Hilbertian fields.

\begin{question}\label{q:special_Hilbertian}
Does every smooth projective special variety $X$ over a Hilbertian field $K$ of characteristic zero
have potential WHP?
\end{question}

\section{Preservation theorems}
\label{sec:preservation}

\noindent
In this section we discuss the preservation of Hilbert properties
under various algebraic and geometric constructions.
We always let $K$ be a field of characteristic zero.

\subsection{Base change}
\label{sec:basechange}

Already Hilbert observed that his theorem for $\mathbb{Q}$
implies the same statement over arbitrary number fields.  In fact, any finite extension of a Hilbertian field is again Hilbertian.
Similarly, HP and WHP are preserved under finite field extensions.  
We now give some examples of what is known regarding infinite extensions (which includes references for the case of finite extensions).
Let 
$X$ be a normal $K$-variety and let
$L/K$ be a field extension.

\begin{theorem}\label{thm:permanence}
If $K$ is Hilbertian, then so is its extension $L$,
whenever one of the following conditions holds:
\begin{enumerate}
    \item $L/K$ is finitely generated.
    \item $L/K$ is small, i.e.~$L/K$ is algebraic and for every $n\in\mathbb{N}$ there are only finitely many intermediate fields $K\subseteq F\subseteq L$ with $[F:K]\leq n$.
    \item There are Galois extensions $M_1/K$ and $M_2/K$ such that
    $L\subseteq M_1M_2$, $L\not\subseteq M_1$, $L\not\subseteq M_2$.
     \item There are $m\in\mathbb{N}$ and fields $K=K_0\subseteq \dots\subseteq K_m$ with $L\subseteq K_m$
    such that for each $i$, $K_i/K_{i-1}$ is Galois with Galois group either abelian 
    or a (possibly infinite) product of finite simple groups.
    \item  
    $L/K$ is algebraic and there is $n\in\mathbb{N}$ and for each prime number $\ell$ a Galois representation (i.e. a continuous homomorphism)
    $\rho_\ell:{\rm Gal}(\overline{K}/K)\rightarrow{\rm GL}_n(\overline{\mathbb{Q}}_\ell)$  such that
    $L$ is fixed by $\bigcap_{\ell}{\rm ker}(\rho_\ell)$.
\end{enumerate}
\end{theorem}

\begin{proof}
(1) is classical, see e.g.~\cite[Propositions 12.3.3 and 13.2.1]{FJ}.
(2) is written down in the special case where $L/K$ is Galois in \cite[Proposition 16.11.1]{FJ};
it applies in particular when ${\rm Gal}(L/K)$ is finitely generated as a profinite group;
the general statement is a special case of Theorem \ref{thm:base_change}(2) below.
(3) is \cite[Theorem 4.1]{Haran}; 
it applies in particular when $L$ is a proper finite extension of a Galois extension of $K$, which is \cite[Satz 9.7]{Weissauer}.
(4) is \cite[Theorem 1.1]{BFW}; the special case where $L/K$ is abelian is in \cite{Kuyk},
and the special case where $L/K$ is Galois with Galois group a product of finite simple groups follows from (3)
and is also contained in \cite{Kuyk}.
(5) builds on (4) and is \cite[Theorem 1.2]{BFW}; 
taking the Galois representation on the $\ell$-adic Tate module it immediately applies to the case when $L$ is contained in the division field $K(A_{\rm tor})$ of an abelian variety $A/K$, cf.~\cite[Theorem 1.3]{BFW}; this case was posed as a question in \cite{Jarden}, and partial solutions were given
in \cite{Jarden, FJP,Thornhill}.
\end{proof}

\begin{theorem}\label{thm:base_change}
If $X$ has HP, respectively WHP, then so does the base change $X_L$,
whenever one of the following conditions holds:
\begin{enumerate}
    \item $L/K$ is finitely generated.
    \item $L/K$ is small (cf.~Theorem \ref{thm:permanence}(2)).
\end{enumerate}   
\end{theorem}

\begin{proof}
(1) is \cite[Corollary 3.3]{BFP24} and (2) is \cite[Corollary 4.3]{BFP24}.
The special case (of both (1) and (2)) where $L/K$ is finite is \cite[Corollary 3.2.2]{Serre} for HP
and \cite[Proposition 3.15]{CDJLZ} for WHP (assuming $X$ is smooth and proper).
\end{proof}

\begin{question}
If $X$ has HP, respectively WHP, over $K$,
and $L$ is as in (3)-(5) of Theorem \ref{thm:permanence}, does it follow that
$X_L$ has HP, respectively WHP?
\end{question}

For partial results concerning WHP analogues of (4) and (5) of Theorem \ref{thm:permanence}
for abelian varieties 
see Section \ref{sec:AV}.

\begin{remark}\label{remark:preservation}
The results of Theorem \ref{thm:base_change} have stronger and more precise formulations in terms of thin and strongly thin sets, see \cite{BFP24}.
Similarly, several of the results on Theorem \ref{thm:permanence} have stronger formulations in terms of
so-called Hilbert sets, as defined in \cite[Section 12.1]{FJ}.
The connection between these two notions is that a subset of $\mathbb{A}^n(K)$ is thin if and only if its complement contains a Hilbert set.
\end{remark}

Contrary to these positive results,
 $X_L$ never has HP or WHP in the case $L=\overline{K}$ and ${\rm dim}(X)>0$ (Proposition \ref{prop:WHP_implies_Hilbertian}),
so some condition on $L/K$ as in Theorem \ref{thm:permanence} is necessary.
Also, if $L/K$ is a finite extension and $L$ is Hilbertian, then it is not necessarily the case that $K$ is Hilbertian,
cf.~Example~\ref{ex:PAC}.
In particular, $\mathbb{P}^1_K$ does not have HP in this case, whereas $\mathbb{P}^1_L$ does.

\subsection{Going down morphisms}

Under certain conditions on a morphism, 
HP and WHP transfer from the source to the target.

\begin{theorem}\label{thm:down} 
Let $f\colon X\rightarrow Y$ be a morphism of normal $K$-varieties.
\begin{enumerate}
\item If $X$ has HP
and $f$ is dominant
and has geometrically irreducible generic fiber,
then  $Y$ has HP.
\item  If $X$ has WHP and
$f$ is smooth, surjective, and has geometrically irreducible generic fiber,
then  $Y$ has WHP.  
\item If $X$ has WHP 
and $f$ is smooth and proper,
then $Y$ has WHP.
\end{enumerate}
\end{theorem}

\begin{proof}
(1) is \cite[Proposition 7.13]{CTS}, 
(2) and (3) follow from parts (a) respectively (b) of \cite[Lemma 5.3]
{BFP24}, phrased for strongly thin sets.
In the case where $X$ and $Y$ are smooth and proper, (2) and (3) follow from \cite[Theorem 3.7]{CDJLZ}.
\end{proof}

\begin{remark}
As in Remark \ref{remark:preservation}, there are more precise versions of Theorem \ref{thm:down} for thin and strongly thin sets in the literature,
see \cite[Remark 5.4]{BFP24}.
Related to Theorem \ref{thm:down}(3), Conjecture \ref{conj:campana_corvaja_zannier}
predicts that if
$K$ is finitely generated,
$X\to Y$ is a surjective morphism of smooth projective $K$-varieties 
and $X$ has potential WHP, then $Y$ has potential WHP.
A variety $Y$ for which there exists $f$ as in (1) with $X$ rational is called
{\em Hilbert-unirational} in \cite{DemeioStreeterWinter}, following a suggestion by Colliot-Th\'el\`ene.
\end{remark}
 
As a particular case of ``going down'' a morphism,
HP passes to {\em quotient} varieties in some situations.
We will see more results on quotient varieties in Section \ref{sec:algebraic_groups}.

\begin{theorem}  
Let $X$ be a geometrically integral $K$-variety
and $G$ a finite group acting generically freely on $X$.
Assume there exists a linearly disjoint family of Galois extensions $(L_i)_{i\in\mathbb{N}}$ of $K$
such that for each $i\in \mathbb{N}$ there exists
an isomorphism $\alpha_i\in{\rm Hom}({\rm Gal}(L_i/K),G)$
such that the twist of $X$ by $\alpha_i$ has HP.
Then the quotient variety $X/G$ has HP.
\end{theorem}

\begin{proof}  
This is \cite[Theorem 3.1]{Demeio}, except that condition (ii) there is replaced 
by the equivalent condition of linear disjointness, cf.~\cite[Remark 3.4]{Demeio}.
\end{proof}

For example, using Shafarevich's theorem,
it follows that $\mathbb{A}^n_\mathbb{Q}/G$ has HP for every 
linear action of a finite solvable group $G$ on $\mathbb{A}^n_\mathbb{Q}$ \cite[Corollary 3.5]{Demeio}\footnote{The proof of \cite[Corollary 3.5]{Demeio} uses \cite[Proposition 1.2]{Demeio}, which misses the assumption that the field is Hilbertian, but this assumption is anyway satisfied in the situation of \cite[Corollary 3.5]{Demeio}.}.

\subsection{Going up morphisms}
In the other direction, potential WHP transfers also from the target to source in some cases:

\begin{theorem}\label{thm:up}
Let $K$ be finitely generated 
and let $f\colon X\rightarrow Y$ be an \'etale cover of smooth proper geometrically integral $K$-varieties.
If $Y$ has potential WHP, then so does $X$.
\end{theorem}

\begin{proof}
This follows from \cite[Theorem 3.16]{CDJLZ},
which uses a Chevalley--Weil theorem over finitely generated fields \cite[Theorem 3.8]{CDJLZ}.
\end{proof}

\begin{remark} 
In Theorem \ref{thm:up}, one can drop neither ``étale'' nor ``potential''. 
For the latter see the following Example \ref{ex:torsor_of_elliptic_curve},
which shows that WHP does not transfer from the target to the source, even if the cover is étale.
For the former take any smooth projective curve $X$ over $\mathbb{Q}$ of genus $g_X>1$ and $f\colon X\rightarrow\mathbb{P}^1_\mathbb{Q}$ a nonconstant rational function on $X$; then  
$\mathbb{P}^1_\mathbb{Q}$ has HP (in particular potential WHP) but $X$ does not have potential WHP, see Theorem~\ref{thm:curves} below.
In particular, HP does not transfer from the target to the source.
However, Theorem~\ref{thm:up} with ``WHP'' replaced by ``HP'' still holds,
since the étale cover $f$ is necessarily trivial in this case by Theorem \ref{thm:CZ}.
\end{remark}
 
\begin{example}\label{ex:torsor_of_elliptic_curve}
If $Y=E$ is an elliptic curve over $\mathbb{Q}$ with nontrivial 2-torsion
and $X$ is a homogeneous space for $E$,
then there exists an \'etale cover $X\rightarrow E$ \cite[Proposition X.4.9]{Silverman}.
However, if $E$ has positive Mordell--Weil rank but $X(\mathbb{Q})=\emptyset$,
then $E$ has WHP (Example \ref{ex:E_WHP}) but $X$ does not
(although it has potential WHP by Example \ref{ex:E_WHP}, in agreement with Theorem \ref{thm:up}).
A concrete example for such $E$ and $X$ can be found in \cite[Example X.4.10]{Silverman}.
\end{example}
 
It is unclear whether the hypothesis that $K$ is finitely generated in Theorem \ref{thm:up} is necessary:

\begin{question} 
Let $K$ be Hilbertian
and let $f\colon X\to Y$ be an \'etale cover of smooth proper geometrically integral $K$-varieties.
Does $X$ have potential WHP whenever $Y$ has potential WHP?
\end{question}

\subsection{Product and fibration theorems}

Serre asked in \cite[p.~20]{Serre} whether the product of two varieties with HP again has HP.
This led to a number of results on products:

\begin{theorem}\label{thm:product}
Let $X$ and $Y$ be normal $K$-varieties.
\begin{enumerate}
    \item If $X$ and $Y$ have HP, then so does $X\times Y$. 
    \item Assume that $K$ is finitely generated and $X$ and $Y$ are proper over $K$. If $X$ and $Y$ have WHP, then so does $X\times Y$.  
\end{enumerate}
\end{theorem}

\begin{proof}
(1) was first published in \cite{BFP};
an unpublished manuscript of van den Dries from the 1990's proves this using nonstandard methods;
yet another proof was given in \cite[Corollary 3.5]{BFP24};
for $K$ a number field, (1) also follows from \cite[Lemma 8.12]{HW16}.
(2) follows from \cite[Theorem 1.4]{Luger2}, generalizing the special case of smooth proper varieties in \cite[Theorem 1.9]{CDJLZ} due to the second-named author of this survey and Wittenberg. 
\end{proof}

Note that the converse holds: If $X\times Y$ has HP or WHP, then so do both $X$ and $Y$.
If $X$ and $Y$ are smooth or $X\times Y$ has HP,
this follows from Theorem \ref{thm:down}
applied to the projections
$X\times Y\rightarrow X$ and $X\times Y\rightarrow Y$.
To see that it holds in general note that every ramified cover $X'\rightarrow X$ gives rise to a ramified cover $X'\times Y\rightarrow X\times Y$.

It is not known whether the assumption that $K$ is finitely generated in Theorem \ref{thm:product}(2) can be dropped.  
\begin{question}\label{q:product}
Let $K$ be Hilbertian 
and let $X$ and $Y$ be normal $K$-varieties. 
Does $X\times Y$ have WHP if and only if both $X$ and $Y$ have WHP?
\end{question}
 
Most product theorems for HP and WHP in the literature are deduced 
from fibration theorems, like the following:

\begin{theorem}\label{thm:mixed_fibration}\label{thm:fibration}
Let $f\colon X\rightarrow S$ be a dominant morphism of normal $K$-varieties
and $\Gamma\subseteq X(K)$.
Let $\Sigma\subseteq S(K)$ be such that for every $s\in\Sigma$, the fiber $X_s=f^{-1}(s)$
is  integral and  $\Gamma\cap X_s$ is not thin in $X_s$.
\begin{enumerate}
    \item If $\Sigma$ is not thin in $S$, then $\Gamma$ is not thin in $X$.
    \item If $\Sigma$ is not strongly thin in $S$, then $\Gamma$ is not strongly thin in $X$.
\end{enumerate}
In particular, if there is a dense open subset $U\subseteq S$ such that for every $s\in U(K)$ the fiber $X_s$ has HP,
then HP or WHP for $S$ imply HP respectively WHP for $X$.
\end{theorem}
\begin{proof}
(1) is \cite[Theorem 1.7]{Luger}, improving on a fibration result for HP \cite[Theorem 1.1]{BFP}.    
(2) is \cite[Theorem 2]{Luger3},
building on \cite[Theorem 1.3]{Javanpeykar22}.
The ``in particular'' part follows by setting $\Sigma=U(K)$ and $\Gamma=X(K)$.
\end{proof}

\begin{theorem}\label{thm:generic_fiber_fibration}
Let $f\colon X\rightarrow S$ be a dominant morphism of normal $K$-varieties,
and assume that the generic fiber $Z$ of $f$ is integral.
\begin{enumerate}
\item If $S$ has HP and $Z$ has HP, then $X$ has HP.
\item If $S$ has HP, $f$ is smooth and proper, and $Z$ has WHP, then $X$ has WHP.
\end{enumerate}
\end{theorem}

\begin{proof}
(1) is \cite[Proposition 3.4]{BFP24},
and (2) is \cite[Theorem 1.6]{Pet25}.
\end{proof}

\begin{question}\label{question:fibrations}
Does Theorem \ref{thm:fibration}(2) hold if
we add the assumption that $X_s$ is normal but
$\Gamma\cap X_s$ is only assumed ``not strongly thin'' instead of ``not thin'' (for every $s\in\Sigma$)?  
Does Theorem \ref{thm:generic_fiber_fibration}(2) still hold if $S$ is only assumed to have WHP?
\end{question}

Note that a positive answer to 
any of these two questions
would also imply 
a positive answer to Question \ref{q:product},
by applying it to the projection $X\times Y\rightarrow Y$
(in the case of the second question 
this uses Theorem \ref{thm:base_change}(1)
and requires $X$ to be smooth).

\section{Curves}
\label{sec:curves}
 
\noindent
We summarize what is known regarding HP and WHP for curves.
Let $K$ always be a field of characteristic zero.
By a {\em curve} $X$ over $K$ we mean
a smooth projective geometrically integral variety over $K$ of dimension one.
As usual, the {\em genus} of $X$ is $g_X={\rm dim}_KH^1(X,\mathcal{O}_X)$. 
For general $K$,
only statements about curves of genus zero can be made:

\begin{proposition}\label{prop:curve0}
Let $K$ be Hilbertian
and let $X$ be a 
curve over $K$
of genus $g_X=0$.
The following are equivalent:
\begin{enumerate}
    \item $X$ has HP.
    \item $X$ is rational, i.e.~$X\cong\mathbb{P}_K^1$.
    \item $X(K)\neq\emptyset$.
\end{enumerate}
\end{proposition}

\begin{proof}
The equivalence $(2)\Leftrightarrow(3)$ holds over every $K$,
$(2)\Rightarrow(1)$ holds because $K$ is Hilbertian
(Proposition~\ref{prop:WHP_implies_Hilbertian}),
and $(1)\Rightarrow(3)$ holds by definition.
\end{proof}

However, the case where $K$ is finitely generated is completely understood:
 
\begin{theorem}\label{thm:curves}
Let $K$ be finitely generated
and let $X$ be a 
 curve over $K$.
 \begin{enumerate}
\item $X$ has HP if and only if $X$ is rational. 
\item  $X$ has WHP if and only if $X$ is rational, or $g_X\leq 1$ and  $X(K)$ is Zariski-dense in $X$.  
\item $X$ has potential HP if and only if $g_X=0$.
\item $X$ has potential WHP if and only if   $g_X\leq 1$.
\end{enumerate}
\end{theorem}

\begin{proof}
Since $K$ is Hilbertian (Theorem \ref{thm:Hilbertian}), the case $g_X=0$ is handled by Proposition \ref{prop:curve0}.
By Faltings's theorem \cite{Faltings}, if $g_X>1$, then $X(L)$ is finite for every finite field extension $L/K$.
In particular, 
if $g_X>1$, then $X$ does not even have potential WHP. 
So assume now that $g_X=1$.
Then $X$ does not have HP (cf.~Example \ref{ex:abelian_var_not_HP}),
but if $X(K)$ is Zariski-dense in $X$, then $X$ has WHP,
since every ramified cover of $X$ is a curve of genus $>1$ (cf.~Example \ref{ex:E_WHP} and \cite[\S2.1]{CZ}). For a proof not involving Faltings's theorem see \cite[Theorem 2]{Zannier}.
Finally, there always is a finite extension $L/K$ with $X(L)$ Zariski-dense
\cite{FreyJarden}.  
\end{proof}

\begin{remark}\label{rem:curve_Faltings}
As mentioned already in Example \ref{ex:E_WHP},
the proof using Faltings's theorem 
in the case $g_X=1$
shows in fact more:
Every infinite subset of $X(K)$ is not strongly thin in $X$.
In the notation of Theorem~\ref{thm:curves}, 
the previous proof shows that 
a not necessarily projective curve $X_0\subseteq X$
has HP, respectively WHP, if and only if $X$ does.  
\end{remark}
 
\begin{remark}
Let $K$ be Hilbertian and $X$ a 
curve over $K$.
To see that the statement of Theorem \ref{thm:curves} does not hold 
in this generality,
recall that if $K$ is also PAC,
all such $X$ have HP (Example \ref{ex:PAC}),
so the implication $\Rightarrow$ in (1)-(4) does not hold.
The following Example \ref{ex:E_over_Ct} shows that also the implication $\Leftarrow$ in (2) does not hold.
However, 
Proposition \ref{prop:curve0} implies that $\Leftarrow$ in (1) and (3)
still hold,
and a positive answer to Question \ref{q:special_Hilbertian} would imply
that also $\Leftarrow$ in (4) holds in general.
\end{remark}

\begin{example}\label{ex:E_over_Ct}
Let $E$ be an elliptic curve over $\mathbb{C}$, 
$K=\mathbb{C}(t)$
and $X=E_K$.
Then $K$ is Hilbertian (Theorem~\ref{thm:Hilbertian}),
$X(K)$ is Zariski-dense in $X$,
but $X$ does not have WHP:
As every rational map $\mathbb{P}^1_K\dashrightarrow X$ is constant, we have that
$X(K)=E(\mathbb{C})$. 
If $\pi\colon Y\to E$ is any ramified cover, then so is $\pi_K\colon Y_K\to X$,
and $\pi_K(Y_K(K))\supseteq \pi(Y(\mathbb{C}))=E(\mathbb{C})=X(K)$,
so $X(K)$ is strongly thin in $X$.
\end{example}

For elliptic curves, 
some positive results have been obtained over other Hilbertian fields like $\mathbb{Q}^{\rm ab}$, see Theorem \ref{thm:bfp} below.

\section{Algebraic groups}
\label{sec:AV}
\label{sec:algebraic_groups}  

\noindent
We now turn to algebraic groups
and discuss what is known regarding the (strong) thinness of certain subsets of rational points.
Let again $K$ be a field of characteristic zero.
An \emph{algebraic group} over $K$ is a group scheme of finite type over $K$. Every algebraic group is smooth and quasi-projective over $K$. 
We again first collect what is known regarding HP and WHP
over Hilbertian fields and over finitely generated fields.
 
 \begin{theorem}\label{thm:alg_groups_1}
     Let $K$ be Hilbertian and $G$ a connected  linear algebraic group over $K$.
     Then $G$ has HP.
 \end{theorem}
 
 \begin{proof}
This is \cite[Theorem 4.2]{BFP}, building on the case of reductive $G$ in \cite[Cor.~7.15]{CTS};
in the case where $K$ is a number field,
\cite[Cor.~3.5(ii)]{Sansuc} proves that $G$ even has WWA (cf.~Theorem~\ref{thm:WWA}).
 \end{proof}

\begin{theorem}\label{thm:alg_groups_2}
Let $K$ be finitely generated 
and let $G$ be a connected algebraic group over $K$.
\begin{enumerate}
  \item $G$ has HP if and only if $G$ is linear.
  \item $G$ has WHP if and only if $G(K)$ is Zariski-dense in $G$.
  \item $G$ has potential WHP.
\end{enumerate}
\end{theorem}

\begin{proof}
There is a connected linear algebraic normal subgroup $H\subseteq G$ such that $A=G/H$ is an abelian variety (see \cite[Thm.~1.1]{Conrad}). 
Since the morphism $G\to G/H$ is dominant with geometrically irreducible generic fibre,  we see that  (1) follows from Theorem \ref{thm:alg_groups_1} and Theorem \ref{thm:down}(1) combined with Example \ref{ex:abelian_var_not_HP}, see \cite[Corollary 4.7]{BFP} for the case where $K$ is a number field.    

Next, (2) follows from  \cite[Theorem 1.6]{Luger2}, 
which builds
on \cite{Liu},
and on \cite{CDJLZ},
which in particular proves the special case where $G$ is an abelian variety;
this special case was posed as a question in \cite{CZ}, and partial results had been obtained in 
\cite{Zannier} and \cite{Javanpeykar}.
 
Finally, (3) follows from (2): 
$H(K)$ is Zariski-dense in $H$ (as follows e.g.~from (1))
and $A(L)$ is Zariski-dense in $A$ for some finite extension $L/K$ \cite{FreyJarden},
so since $A(L)$ is finitely generated by the theorems 
 of Mordell--Weil and Lang--N\'eron, there is a finite extension $L'/L$ such that the image of $G(L')$ in $A(L')$ contains $A(L)$,
 and then $G(L')$ is Zariski-dense in $G$.
\end{proof}

\begin{remark}
Since connected algebraic groups are special (Example \ref{ex:special}), 
Theorem \ref{thm:alg_groups_2}(3) confirms a case of
Conjecture \ref{conj:campana_corvaja_zannier}.
A positive answer to Question \ref{q:special_Hilbertian}
would imply that (3) holds for arbitrary Hilbertian $K$,
which in the case where $G$ is an abelian variety was confirmed for function fields $K$ in \cite[Theorem A]{Jav25}.
However, Example \ref{ex:E_over_Ct}
shows that already in the case of elliptic curves,
(2) does not hold even for function fields $K$. 
\end{remark}

In some cases, HP for linear algebraic groups carries over to homogeneous spaces for these groups:

\begin{theorem} \label{thm:homogeneous}
Let $G$ be a connected linear algebraic group over $K$,
and $H\leq G$ an algebraic subgroup.
\begin{enumerate}
\item The quotient variety $G/H$ has HP in each of the following cases:
\begin{enumerate}
\item $K$ is Hilbertian and $H$ is connected.    
\item $K$ is a number field, $G$ is simply connected, and $H$ is abelian.
\end{enumerate}
\item Moreover, $G/H$ has WHP whenever $K$ is Hilbertian.
\end{enumerate}
\end{theorem}

\begin{proof}
(1a) is \cite[Corollary 4.6]{BFP}, see also \cite[Proposition 1.1]{Borovoi}.
(1b) is \cite[Theorem 1.2]{Borovoi}; in fact, \S4 of that paper contains a second proof, which shows that $G/H$ has WWA.
For (2), note that if $H^\circ$ denotes the connected component of the identity of $H$, then $G/H^\circ$ has HP by (1a), in particular WHP,
and since $G/H^\circ\rightarrow G/H$ is an \'etale cover,
Theorem \ref{thm:down}(3) implies that also $G/H$ has WHP;
the case where $K$ is finitely generated is also proven in
\cite[Theorem 1.8]{Liu}.
\end{proof}

Finally, as already alluded to in Section \ref{sec:curves}, 
the weak Hilbert property for abelian varieties is in some cases preserved under base change in extensions beyond those of Theorem \ref{thm:base_change}: 

\begin{theorem}\label{thm:bfp}   
Let $K$ be a number field and $A$ an abelian variety over $K$.
\begin{enumerate}
    \item If $L=K^{\rm ab}$, then $A_{L}$ has WHP in each of the following cases:
    \begin{enumerate}[$(a)$]
     \item ${\rm dim}(A)=1$
     \item $A$ is geometrically simple and $A(L)$
    has infinite rank (i.e.~${\rm dim}_\mathbb{Q}(A(L)\otimes_\mathbb{Z}\mathbb{Q})=\infty$).
    \item The Frey--Jarden conjecture holds: Every abelian variety over $\mathbb{Q}^{\rm ab}$ has infinite rank.
    \end{enumerate}
    \item If $L=K(B_{\rm tor})$ for an abelian variety $B$ over $K$, then $A_L$ has WHP in the following cases:
    \begin{enumerate}[$(a)$]
       \item $K=\mathbb{Q}$ and ${\rm dim}(A)=1$ 
       \item $A$ is geometrically simple and $A(K\mathbb{Q}^{\rm ab})$
    has infinite rank.
       \item The Frey--Jarden conjecture holds.
    \end{enumerate}
\end{enumerate}
\end{theorem}

\begin{proof}
$(1c)$ is \cite[Corollary 1.5]{BFP23}, 
$(1b)$ follows from \cite[Theorem~1.4]{BFP23},
and $(1a)$ follows from $(1b)$ and e.g.~\cite[Theorem 1.1]{Petersen};
all of these use the stronger results from \cite{CDJLZ}, 
see Theorem \ref{thm:alg_groups_fingen} below.
$(2b)$ is \cite[Theorem 5.1]{GP},
which uses \cite{BFP23} and \cite{BFP24};
basically the same proof also shows $(2c)$;
$(2a)$ follows from $(2b)$ and was first proven in \cite[Theorem 1.7]{BFP23}.
\end{proof}

The results up to here
all said that if $G$ is an algebraic group over a field $K$,
then under certain assumptions on $K$ and $G$, $G(K)$ is not thin or not strongly thin.
We now discuss results for subsets of $G(K)$.

\begin{theorem}\label{thm:alg_groups_fingen}
Let $K$ be finitely generated, let $G$ be a connected algebraic group over $K$,
and let $\Gamma\leq G(K)$ be a Zariski-dense finitely generated subgroup.
For $i=1,\dots,r$ let $\pi_i\colon Y_i\rightarrow G$ be a cover that does not dominate a nontrivial \'etale cover of $G$.
Assume that
\begin{enumerate}
    \item $G$ is linear, or
    \item $G$ is an abelian variety.
\end{enumerate}
Then there exists a coset $\Omega$ of a finite index subgroup of $\Gamma$ such that the fiber $\pi_i^{-1}(x)$ is irreducible for each $i=1,\dots,r$ and each $x\in\Omega$.
In particular, $\Gamma$ is not strongly thin in $G$.
\end{theorem}

\begin{proof}
(1) is \cite[Theorem 1.3]{Liu}, and 
(2) is \cite[Theorem 1.4]{CDJLZ}.
For cyclic $\Gamma$ and $K$ a number field,
(1) in the case $G=(\mathbb{G}_m)^r$ and (2) in the case $G=E^r$ for an elliptic curve $E$ was proven in \cite{Zannier},
which for $G=(\mathbb{G}_m)^r$ built on results from \cite{DZ} over $\mathbb{Q}^{\rm ab}$, and for $G=E^r$ on ideas from \cite{Zannier_Pisot}.
For a quantitative version in case (1) phrased in terms of random walks see \cite{BSG}.

The ``in particular'' part
in case (1) is \cite[Corollary 1.4]{Liu},
with the special case $G=(\mathbb{G}_m)^r$ and $\Gamma$ cyclic
already in \cite[Theorem 4]{Zannier}.
For case (2),   
see \cite[Proposition 4.17]{CDJLZ} and \cite[Theorem~1.3]{CDJLZ}.
\end{proof}

For subsets of integral points (see Section \ref{sec:integral}), the following question arises:

\begin{question}\label{q:alg_groups_integral}
Let $K$ be a number field,
$S$ a finite set of finite places of $K$,
and $\mathcal{G}$ a connected group scheme over $\mathcal{O}_{K,S}$ of finite type.
Is $\mathcal{G}(\mathcal{O}_{K,S})$
strongly thin in $\mathcal{G}_K$ if and only if it is not Zariski-dense?
\end{question}

Note that Theorem \ref{thm:alg_groups_fingen} gives a positive answer to Question \ref{q:alg_groups_integral} whenever $\mathcal{G}_K$ is linear (or an abelian variety, but in that case $\mathcal{G}(\mathcal{O}_{K,S})=\mathcal{G}(K)$).
The current best result towards a general answer is the following:

\begin{theorem}\label{thm:potential_WHP_algebraic_groups}
    Let $G$ be a connected algebraic group over a number field $K$ with $G(K)$ dense.  Then there is a finite set of finite places $S$ of $K$ and a model $\mathcal{G}$ for $G$ over $\mathcal{O}_{K,S}$ such that $\mathcal{G}(\mathcal{O}_{K,S})$ is not strongly thin in $G$. 
\end{theorem}

\begin{proof}
This is
\cite[Theorem~1.6]{Luger2},
using \cite{Liu} (which itself builds on \cite{CorvajaAlgebraicGroups}) and \cite{CDJLZ}.   
\end{proof}

\begin{remark} 
In the case where $G$ is linear, \cite[Theorem~1.3]{Luger4} proves
that the statement of Theorem~\ref{thm:potential_WHP_algebraic_groups} still holds
with $G$ replaced by the complement of a closed subset of codimension at least two, after a finite extension of the base field $K$.
As the complement of a closed subset of codimension at least two in a smooth special variety is still special \cite[Theorem~G]{BJL},
this is in line with the quasi-projective extension of Conjecture~\ref{conj:campana_corvaja_zannier} (cf.~\cite[Conjecture 1.12]{BJL}),
which predicts in particular that
if the rational points on a  smooth projective variety are not thin (resp.~not strongly thin), then the integral points on any open subset with complement of codimension at least two are also not thin  (resp.~not strongly thin), after possibly extending the base field and choosing an appropriate model.
This would follow in certain cases from a positive answer to Wittenberg's question on varieties with strong approximation \cite[Question~2.11]{Wittenberg}.  
Similarly, the fact that  potential density of integral points of a smooth variety should be inherited by open subsets with complement of codimension at least two is also predicted by Hassett--Tschinkel's puncturing problems \cite[Problem~2.11 and Problem~2.14]{HT01}. 
\end{remark}

Motivated by the case of elliptic curves over number fields (Remark \ref{ex:E_WHP}), in the situation of Theorem~\ref{thm:alg_groups_fingen}(2), even smaller subsets should be not strongly thin:

\begin{question}\label{q:abelian_dense_strongly_thin}
Let $A$ be an abelian variety over a number field $K$.
Is every Zariski-dense subset of $A(K)$ not strongly thin in $A$?
\end{question}

\begin{proposition}
Lang's conjecture (Conjecture \ref{conjecture:lang_mordellic})
implies a positive answer to Question \ref{q:abelian_dense_strongly_thin}
\end{proposition}

\begin{proof}  
It suffices to show that $\pi(Z(K))$ is not Zariski-dense in $A$
for any ramified cover $\pi\colon Z\to A$ (with $Z$ normal). 
By \cite[Theorem~26~and~Theorem~27]{Kawamata}, replacing $K$ by a finite field extension if necessary, there is an \'etale cover $Z'\to Z$ and a surjective morphism $Z'\to X$, where $X$ is a normal projective variety 
with $\kappa(X)={\rm dim}(X)=\kappa(Z)>0$.
In particular, $X$ is positive-dimensional and of general type. 
Conjecture \ref{conjecture:lang_mordellic} implies that 
$X(L)$ is not Zariski-dense in $X$ for every finite extension $L/K$,
hence $Z'(L)$ is not Zariski-dense in $Z'$.
By the Chevalley--Weil theorem \cite[Theorem~3.8]{CDJLZ}, it follows that $Z(K)$ is not Zariski-dense in $Z$, hence $\pi(Z(K))$ is not Zariski-dense in $A$. 
\end{proof}

\section{Surfaces}   
\label{sec:surfaces}

\noindent
Let $K$ be a field of characteristic zero
and $X$ a {\em surface} over $K$, by which we always mean
a smooth projective geometrically integral variety over $K$ of dimension two.
We discuss what is known regarding HP and WHP for $X$
depending on the Kodaira dimension $\kappa(X)\in\{-\infty,0,1,2\}$.
We recall that 
the Kodaira dimension is invariant
under birational equivalence and base change \cite[Proposition 9.1.2]{Poonen}
and
$X_{\overline{K}}$ is birationally equivalent to a minimal surface \cite[Theorem~6.3]{Badescu},
to which the Enriques--Kodaira classification applies,
cf.~\cite{Badescu, Beauville, Liedtke}.
 
\subsection{Surfaces with $\kappa(X)=-\infty$}
\label{sec:kappaminusinfty}
   By a theorem of Enriques \cite[Theorem~13.2]{Badescu},  
  $\kappa(X)=-\infty$ if and only if $X_{\overline{K}}$ is {\em birationally ruled},
i.e.~$X_{\overline{K}}\sim\mathbb{P}^1_{\overline{K}}\times C$ 
for a (smooth projective) curve $C$.  
  
The {\em irregularity} $q(X)=\dim_K \mathrm{H}^1(X,\mathcal{O}_X)$ of $X$ is a birational invariant \cite[Remark 6.4]{Badescu} and invariant under base change,
so if $X_{\overline{K}}\sim\mathbb{P}^1_{\overline{K}}\times C$,
then $q(X)=g_C$ \cite[Corollary V.2.5]{Hartshorne}.
 
We start with the most simple observation regarding $q(X)=0$:
\begin{proposition}
If $K$ is Hilbertian, $\kappa(X)=-\infty$ and $q(X)=0$, then $X$ has potential HP. 
\end{proposition}

\begin{proof}
Since $X_{\overline{K}}\sim\mathbb{P}^1_{\overline{K}}\times\mathbb{P}^1_{\overline{K}}$ is rational, there exists $L/K$ finite with $X_L$ rational, and $L$ is Hilbertian,
so $X_L$ has HP (Proposition \ref{prop:WHP_implies_Hilbertian}).    
\end{proof}

Similarly easy, the question whether $X$ has HP or WHP
is well-understood if already $X$ is birationally ruled (and not only $X_{\overline{K}}$):

\begin{proposition}\label{prop:uniruled}
    Let $K$ be finitely generated and let $X$ be a surface that is birationally ruled over $K$.
    \begin{enumerate}
        \item $X$ has HP if and only if $q(X)=0$ and $X(K)$ is Zariski-dense in $X$.
        \item $X$ has WHP if and only if $q(X)\leq 1$ and $X(K)$ is Zariski-dense in $X$.
        \item $X$ has potential HP if and only if $q(X)=0$.
        \item $X$ has potential WHP if and only if $q(X)\leq 1$.
    \end{enumerate}
\end{proposition}

\begin{proof}
Since HP and WHP are birational invariants (of smooth projective varieties, see Remark \ref{rem:HP_birat} and Proposition \ref{prop:WHP_birat}),
we can assume that $X=\mathbb{P}^1_K\times C$ for a curve $C$ over $K$ of genus $g_C=q(X)$.
By Theorem \ref{thm:mixed_fibration} and Theorem \ref{thm:down}
applied to the projection $X\rightarrow C$,
$X$ has HP or WHP if and only if $C$ has.
Moreover, $X(K)$ is Zariski-dense in $X$ if and only if $C(K)$ is Zariski-dense in $C$.
Therefore, all claims follow from Theorem \ref{thm:curves} for the curve $C$.
\end{proof}

\begin{corollary}
Let $K$ be finitely generated and let $X$ be a surface over $K$ with $\kappa(X)=-\infty$. Then $X$ has potential WHP if and only if $q(X)\leq 1$.
\end{corollary}

\begin{proof}
There is a finite extension $L/K$ such that $X_L$ is birationally ruled over $L$,
and $X$ has potential WHP if and only if $X_L$ does,
so the claim follows from Proposition \ref{prop:uniruled}(4).
\end{proof}
 
If $X_{\overline{K}}$ is birationally ruled, but $X$ itself is not,
then (3) and (4) of Proposition \ref{prop:uniruled} still hold,
as do $\Rightarrow$ in (1) and (2),
and we now discuss what can be said regarding $\Leftarrow$ in (1) and (2).

\begin{proposition}\label{prop:q_is_one} 
Let $K$ be finitely generated and $X$ a surface over $K$ with $\kappa(X)=-\infty$ and $q(X)=1$.
If $X(K)$ is Zariski-dense in $X$, then $X$ has WHP.
\end{proposition}

\begin{proof} 
    The Albanese variety $E$ of $X$ (cf.~\cite[Definition~5.2, Theorem~5.3]{Badescu}) is an elliptic curve, since ${\rm dim}(E)=q(X)=1$, 
    and so the Albanese map $f\colon X\to E$ (induced by the choice of a base point in $X(K)$) is surjective. 
    If $X\stackrel{g}{\to} C\stackrel{h}{\to} E$ is the Stein factorization of $f$,
    then $C$ is again an elliptic curve,
    as follows from $g_C\geq g_E$ and $q(X)\geq g_C$,
    and so $C=E$ and $h={\rm id}_E$ by the universal property of the Albanese variety.
    In particular, $f=g$ has geometrically connected fibers. 
    By generic smoothness, almost all fibers of $f$ are also smooth. 
    So since (a dense open subset of) $X_{\overline{K}}$ is covered by rational curves, 
    which are contracted by the Albanese map,
    almost all fibers of $f$ are (smooth projective) curves of genus zero. 
    As $X(K)$ is Zariski-dense in $X$,
    the set $\Sigma$ of $s\in E(K)$ such that $X_s(K)\neq\emptyset$ 
    is infinite
    and therefore not strongly thin in $E$ (see Remark \ref{rem:curve_Faltings}).
    Since $X_s\cong\mathbb{P}^1_K$ has HP for every $s\in\Sigma$ (Proposition \ref{prop:curve0}),
    it follows 
    from Theorem~\ref{thm:mixed_fibration}(2) 
    that $X$ has WHP.
\end{proof}

If $\kappa(X)=-\infty$ and $q(X)=0$, then 
$X_{\overline{K}}$ is rational and so 
$X\sim X'$ where
\begin{enumerate}
 \item $X'$ admits a conic fibration,
 i.e.~a surjective morphism $f\colon X'\rightarrow C$ with 
 both $C$ and the generic fiber of $f$ a (smooth projective) curve of genus zero, or
 \item $X'$ is a del Pezzo surface of degree $d\in\{1,\dots,9\}$, i.e.~the anticanonical divisor $-K_{X'}$ is ample and its self-intersection $K_{X'}^2 $ equals $d$,  
\end{enumerate}
cf.~\cite[Corollary on p.~27]{Iskovskih}.  

\begin{theorem}\label{thm:delPezzosurfaces}
Let $K$ be Hilbertian and $X$ a del Pezzo surface over $K$
of degree $d$
with $X(K)$ Zariski-dense in $X$.
Then $X$ has HP in each of the following cases:
\begin{enumerate}
    \item $d\in\{2,\dots,9\}$
    \item $d=1$ and $X$ admits a conic fibration.
\end{enumerate}
\end{theorem}

\begin{proof}
(1) is \cite[Theorem 1.4(a)]{Streeter} for $d\geq 4$,
\cite[Corollary 3.7]{DemeioStreeterWinter}
for $d=3$
(as every del Pezzo surface of degree 3
is a smooth cubic hypersurface
\cite[Ch.~III, Thm.~3.5]{Kollar_book}), 
and \cite[Theorem~1.3]{DemeioStreeterWinter}  for $d=2$;
if $K$ is a number field, then 
even WWA for $X$ is known
in case $d=4$ 
\cite{SalbergerSkorobogatov}
(see also \cite[Corollary 1.6]{DemeioStreeter}),
$d=3$ \cite{SwinnertonDyer}
and $d=2$ \cite[Theorem 1.1]{DemeioStreeterWinter}.
(2) is \cite[Theorem 1.4(c)]{Streeter}.
\end{proof}

\begin{remark}
As it is believed that every surface $X$ over a number field $K$ with $X_{\overline{K}}$ rational and $X(K)\neq\emptyset$ is unirational
(see \cite[Problem 5.1(b)]{CT_problems}), 
Conjecture \ref{conj:WWA} suggests that such $X$ should satisfy WWA and therefore have HP.
 However, the case where $X$ is a del Pezzo surface of degree $d=1$ without a conic fibration, as well as the case where $X$ admits a conic fibration but is not a del Pezzo surface,
is open,
but in the first case, \cite[Theorem 1.5]{DSW25} 
proves the conjecture for an infinite family.
 \end{remark}

\subsection{Surfaces with $\kappa(X)=0$}
\label{sec:kappa0}

By the Enriques--Kodaira classification,  
if $\kappa(X)=0$,
then $X_{\overline{K}}\sim X'$
with a surface $X'$ over $\overline{K}$
that falls into one of the following three cases :
 
\begin{enumerate}
    \item\label{k0_case1} $X'$ is a K3 surface (i.e., $X'$ is simply connected and $K_{X'}=0\in{\rm Cl}(X')$).
    \item\label{k0_case2} $X'$ is an Enriques surface (i.e.,  there is a finite \'etale cover $Y\to X'$ of degree two with $Y$  a K3 surface over $\overline{K}$).  
  \item\label{k0_case3}  There is a finite \'etale cover $Y\to X'$ with $Y$ an abelian surface over $\overline{K}$. 
\end{enumerate}   
We briefly explain how to deduce this from \cite[Section 10]{Badescu} in the notation used there:
By \cite[Section~10.1, Theorem~10.2]{Badescu}, if $X'$ is minimal and $\kappa(X')=0$, then $(K^2)=0$, $p_g\leq 1$ and $b_2\in\{ 2,6,10,22\}$. 
If $b_2=22$, then $X'$ falls under case (1) \cite[Theorem~10.3]{Badescu}; 
if $b_2=10$, then $X'$ falls under case (2) \cite[Section~10.10, Proposition 10.14]{Badescu};
if $b_2=6$, then $X'$ is an abelian surface and thus falls under case (3) \cite[Theorem~10.19]{Badescu};
and if $b_2=2$, then $X'$ is a hyperelliptic surface \cite[Section~10.20]{Badescu} and thus 
also falls under case (3) \cite[Theorem~10.25]{Badescu}.

\begin{theorem}
Let $K$ be finitely generated and $\kappa(X)=0$.
Then $X$ has potential WHP in cases $(\ref{k0_case2})$ and~$(\ref{k0_case3})$.
\end{theorem}

\begin{proof}
 In case $(\ref{k0_case2})$ this is \cite[Theorem A]{GCM},
which proves the stronger statement that in this case $X$ has WHP whenever $X(K)$ is Zariski-dense
and ${\rm Pic}(X)={\rm Pic}(X_{\overline{K}})$.
  In case $(\ref{k0_case3})$, it follows from Theorem~\ref{thm:down}(3) and Theorem \ref{thm:alg_groups_2}(3).
\end{proof} 

A surface $X$ over $K$ is an \emph{elliptic surface} if there is a curve $C$ over $K$ and a surjective morphism $X\to C$ whose generic fibre is a curve of genus one; we refer to $X\to C$ as an {\em elliptic fibration}.
 
\begin{theorem}
Let $K$ be a number field and $\kappa(X)=0$. 
In case $(\ref{k0_case1})$,
if $X$ admits two distinct elliptic fibrations, then $X$ has potential HP.
\end{theorem}

\begin{proof}
This is  \cite[Theorem 1.1]{GCM2}. 
Prior to that, in  \cite[Prop.~4.4]{Demeio2}, potential HP  was proven for  Kummer  surfaces associated to a product of two elliptic curves, generalizing the example of the Fermat quartic $x^4+y^4 = z^4+ w^4$ in $\mathbb{P}^3_{\mathbb{Q}}$ 
which was shown to have HP over $\mathbb{Q}$ in \cite[Theorem~1.6]{CZ} using two distinct elliptic fibrations; 
this approach of using multiple elliptic fibrations  was 
already used in \cite[Theorem 46.1]{Manin} and 
formalized  
in \cite[Theorem~1.1]{Demeio2}.  In the more general case of finitely generated $K$,
\cite[Theorem B]{GCM} proves the claim if $X$ has Picard rank at most 9, as well as in some other cases.
\end{proof}
 
\begin{remark}
Since all smooth projective varieties with Kodaira dimension zero are special 
(Example~\ref{ex:special}),
Conjecture \ref{conj:campana_corvaja_zannier} predicts that
if $K$ is finitely generated and $\kappa(X)=0$, then $X$ has potential WHP,
and in case $(\ref{k0_case1})$ in fact potential HP,
as K3 surfaces over $\mathbb{C}$ are simply connected.
\end{remark}

\begin{remark}
Another class of K3 surfaces (case $(\ref{k0_case1})$)
is treated in \cite[Theorem~1.3]{GCH},
and we refer to Section \ref{sec:Kumvar} for a higher dimensional generalization.
We summarize the open problems on the arithmetic of a K3 surface $X$ over a number field $K$
that are relevant questions regarding potential HP. 
\begin{enumerate}
 \item Does there exist a finite field extension $L/K$ such that $X(L)$ is Zariski-dense in $X$?  This is not known if $X$ has (geometric)  Picard rank one. 
    \item Suppose that $X$ admits an elliptic fibration or has infinitely many automorphisms. Then, does $X$ have potential WHP? (In this case it is known that there is a finite field extension $L/K$ such that $X(L)$ is Zariski-dense in $X$ \cite{BT, HT00}.)
   \item     Does $X$ satisfy WWA whenever $X(K)\neq\emptyset$ (see \cite[p.~484]{SkoZar})? 
\end{enumerate}
\end{remark} 

\begin{remark}\label{remark:integral_points} Let $K$ be a number field, let $C$ be a smooth plane cubic in $\mathbb{P}^2_K$ and let $X=\mathbb{P}_K^2\setminus C$.
Using conic fibrations on $X$, \cite[Theorem~1.14]{Coccia} shows that there is a finite field extension $L/K$, a finite set of finite places $S$ of $L$, and a model $\mathcal{X}$ for $X_L$ over $\mathcal{O}_{L,S}$ such that $\mathcal{X}(\mathcal{O}_{L,S})$ is not strongly thin in $X$
(i.e., $X_L$ has WIHP, cf.~Section~\ref{sec:integral}). 
 We note that this (non-projective) special surface is an example of a ``log-K3 surface''. 
 In a similar vein, if $D$ is a curve of degree at most three in $\mathbb{P}^2_K$, then  \cite[Theorem~5.2]{Coccia} shows that WIHP holds for the (special) surface $\mathbb{P}^2_K\setminus D$.  
 More results on such surfaces are proven in \cite[Theorem~1.6]{Coccia2} and \cite[Theorem~1.4]{AlessandriLoughran}, 
 and extensions to ``Campana points'' (or ``semi-integral points'') on orbifold  special surfaces are obtained in  \cite{NakaharaStreeter}.  
\end{remark}

\subsection{Surfaces with $\kappa(X)=1$} 
If $X$ has Kodaira dimension one,
then the canonical divisor $K_X$  defines  an elliptic fibration $f\colon X\to C$ where $C$ is a curve of genus $q(X)$; see \cite[Theorem~6.3]{Liedtke}. 
As in the proof of Proposition \ref{prop:uniruled}, one can show that if 
$K$ is finitely generated and
$X$ has potential WHP, then $q(X)\leq 1$.  
For every closed point $c$ of $C$, the fibre $X_c$ is divisor on $X$, and so we may write $X_c = \sum_{i=1}^{r_c} a_i F_i$, where the $F_i$ are the distinct irreducible components of $X_c$,
and we define $m_c:= \mathrm{gcd}(a_1,\dots,a_{r_c})$. 
The elliptic fibration $X\to C$ has \emph{general type orbifold base} if 
\[
2 g_C -2 + \sum_{c\in C \textrm{ closed}} \Big(1-\frac{1}{m_c}\Big) >0.
\] 
Moreover, a surface $X$ of Kodaira dimension one is special if and only if its elliptic fibration does not have a general type orbifold base \cite[Lemma~3.34]{Ca04}.
Therefore, Conjecture \ref{conj:campana_corvaja_zannier}
predicts that if $K$ is finitely generated, then $X$ has potential WHP if and only if its elliptic fibration does not have a general type orbifold base. 
We prove the forward direction of this:
   
\begin{proposition}\label{thm:general_type_orbifold_base}
 Assume that
$K$ is finitely generated, $\kappa(X)=1$
and the canonical elliptic fibration $f\colon X\to C$ has general type orbifold base.
Then $X(L)$ is not Zariski-dense in $X$ for any finitely generated field extension $L/K$.
 In particular, $X$ does not have potential WHP. 
\end{proposition}
\begin{proof}
     Since $f$ has general type orbifold base, it follows from the theory of elliptic fibrations \cite[Lemma~3.34]{Ca04} that there is  a finite field extension $L/K$,  a curve $D$ of genus at least two over $L$,   an \'etale cover $X'\to X_L$ from a surface $X'$ over $L$ with $\kappa(X')=1$ and $q(X') \geq 2$, and an elliptic fibration $X'\to D$. 
     Since $D$ has genus at least two, it follows from Faltings's theorem \cite{Faltings} that $X'(M)$ is not dense for any finitely generated field extension $M/L$. Therefore, by the Chevalley--Weil theorem \cite[Theorem~3.8]{CDJLZ}, it follows that 
       also $X(M)$ is not Zariski-dense in $X$ for such $M$.
\end{proof}

\subsection{Surfaces with  $\kappa(X)=2$}

Finally, if the Kodaira dimension of $X$ is two (hence maximal), then $X$ is of general type (by definition). In particular,   
Lang's conjecture (Conjecture \ref{conjecture:lang_mordellic}) implies that
if $K$ is a number field, then $X$ does not have potential WHP.

\section{Some higher-dimensional varieties}
\label{sec:higher_dim}

\noindent
We conclude this survey with some results on varieties of arbitrary dimension (excluding algebraic groups, which are discussed in Section \ref{sec:algebraic_groups}).
As before, let $K$ always be a field of characteristic zero
and $X$ a variety over $K$.

\subsection{Closed subvarieties of abelian varieties}

 Work of Faltings \cite{Faltings2} 
  implies the following characterization of subvarieties of abelian varieties with potential WHP.

\begin{theorem}\label{thm:closed_subvarieties_of_ab_vars}
Assume that $K$ is finitely generated,
let $A$ be an abelian variety over $K$
and let $X\subseteq  A$ be a closed subvariety
which is geometrically integral. 
Then the normalization $X'$ of $X$ has potential WHP if and only if $X_{\overline{K}}$ is the translate of an abelian subvariety of $A_{\overline{K}}$.
\end{theorem}

\begin{proof}  
If $X$ is the translate of an abelian subvariety of $A$, then $X=X'$ has potential WHP by Theorem~\ref{thm:alg_groups_2}.
Conversely, assume that $X'_L$ has WHP for a finite extension $L$ of $K$. 
Then $X(L)$ is dense in $X$
and therefore $X_L$ is the translate of an abelian subvariety of $A_L$ by 
the special case of the Mordell--Lang conjecture proved in
\cite{Faltings2}.  
\end{proof} 
 
\begin{remark} 
If $X$ as in Theorem \ref{thm:closed_subvarieties_of_ab_vars} has positive dimension, then it does not have potential HP. 
Indeed, otherwise by
Theorem \ref{thm:closed_subvarieties_of_ab_vars}
there exists a finite extension $L/K$
such that $X_L$ has HP and can be endowed with the structure of an abelian variety, contradicting Example \ref{ex:abelian_var_not_HP}.  
\end{remark}
 
\subsection{Fano varieties}  
\label{sec:Fano}
A smooth projective geometrically integral $K$-variety $X$
is a {\em Fano variety} if 
the anticanonical divisor $-K_X$ is ample.
A Fano variety $X$ is special (as it is geometrically rationally connected, cf.~Example \ref{ex:special}), and $X_{\overline{K}}$ is  simply connected \cite[Corollary~4.29]{DebarreBook}.
Thus,   
if $K$ is finitely generated, then
Conjecture \ref{conj:campana_corvaja_zannier} predicts that $X$ has potential HP.
This is currently not known, unless $\dim X \leq 2$ (in which case $X$ is geometrically rational, cf.~Section~\ref{sec:kappaminusinfty}),
and for the special case
of del Pezzo varieties of degree at least three (see \cite[Definition 4.1]{Streeter} for the definition used here):

\begin{theorem}
Let $X$ be a smooth del Pezzo variety over $K$ of degree $d\geq 3$ and dimension $n\geq2$ with $X(K)\neq\emptyset$. Then $X$ has HP in each of the following cases:
\begin{enumerate}
\item $K$ is Hilbertian and $d\geq 4$.
\item $K$ is Hilbertian, $d=3$, and there exists a line on $X$ (under its anticanonical embedding).
\item $K$ is a number field and $d=3$. 
\end{enumerate}
\end{theorem}
 
\begin{proof}
(1) is \cite[Theorem~1.5(a)]{Streeter}.
(2) is \cite[Theorem~1.5(b)]{Streeter}.
(3) is \cite[Theorem~3.1]{Demeio2},
as $X$ is a smooth cubic hypersurface in this case.
Indeed,
the fundamental class $A_X$ is very ample and $\dim |A_X|=d+n-2=n+1$   \cite[Proposition~2.2(i,iv)]{Kuznetsov},
hence
the embedding $X\to \mathbb{P}^{n+1}$ associated to $|A_X|$ realizes $X$ as a smooth hypersurface of degree $A_X^n=d=3$.  
The special case $n=2$ 
was obtained by a similar argument in \cite[Theorem 46.1]{Manin}, as Colliot-Th\'el\`ene pointed out to us,
and also follows from \cite{SwinnertonDyer}, cf.~Theorem~\ref{thm:delPezzosurfaces}.
\end{proof}

\subsection{Smooth projective varieties with nef tangent bundle} 
Let $K$ be finitely generated and
$X$ smooth and projective.
Specialness is a ``positivity'' condition on the tangent bundle $T_X$ of $X$. For example, if $T_X$ is ample, trivial, or nef, then $X$ is special. In the two extreme cases that $T_X$ is trivial or ample, the situation is well-understood:

\begin{enumerate}
    \item If $T_X$ is ample, then $X$ has potential HP. Indeed, by Mori's celebrated proof of Hartshorne's conjecture \cite{Mori}, we have that $X_{\overline{K}}$ is isomorphic to $\mathbb{P}^n_{\overline{K}}$, so that it follows from Theorem \ref{thm:Hilbertian} that $X$ has potential HP.
    \item If $T_X$ is trivial, then $X_{\overline{K}}$ has the structure of an abelian variety. In particular, $X$ has potential WHP by Theorem \ref{thm:alg_groups_2}.
\end{enumerate}

If $T_X$ is nef, then the geometry of $X$ is not fully understood yet,
but \cite{DPS} shows that
the Albanese map $X\to \mathrm{Alb}(X)$ is smooth proper surjective and its fibres are Fano varieties with nef tangent bundle.  
     
 \begin{theorem}
     Assume that Fano varieties with nef tangent bundle are homogeneous spaces
     (the Campana--Peternell conjecture \cite[p.~170]{CampanaPeternell}). 
     If $K$ is a number field and $X$ is a smooth projective variety over $K$ with $T_X$ nef, then $X$ has potential WHP.
 \end{theorem}

 \begin{proof}
This is \cite[Theorem 1.8]{Javanpeykar22}, which uses
Theorem \ref{thm:mixed_fibration}(2) and Theorem \ref{thm:alg_groups_fingen}(2).
 \end{proof}

\subsection{Kummer varieties}
\label{sec:Kumvar}

If $A$ is an abelian variety, then the Kummer variety $\mathrm{Kum}(A)$ associated to $A$ is the quotient of $A$ by the natural involution on $A$ (given by multiplication with $-1$). 
Such varieties are special (Example \ref{ex:special} and Proposition \ref{thm:special}(2)) and simply connected, and should therefore have potential HP if $K$ is finitely generated,
according to Conjecture \ref{conj:campana_corvaja_zannier}.
In this direction we have the following result:

\begin{theorem} 
If $K$ is finitely generated,
$H$ is a (smooth projective) $K$-curve which admits a degree two morphism to $\mathbb{P}^1_K$,
and $\mathrm{Alb}^1(H)$ denotes the Albanese torsor of $H$,
then $\mathrm{Kum}(\mathrm{Alb}^1(H))$ has HP.
\end{theorem}

\begin{proof}
This is \cite[Theorem~1.5]{GCH}.     
\end{proof}

\subsection{Symmetric products}   
Let $X$   be smooth and  projective.
Similar to $\mathbb{A}^n_K/G$ in Section \ref{sec:HIT},
the $n$-th symmetric power of $X$ is the quotient variety $\mathrm{Sym}^n(X) = X^n/S_n$, which is 
a normal projective variety of dimension $n\cdot \dim(X)$.

\begin{example}
If $X$ is rational over $K$ of positive dimension and has HP,
then $K$ is Hilbertian (Proposition~\ref{prop:WHP_implies_Hilbertian})
and $\mathrm{Sym}^n(X)$ is rational \cite{Mattuck}, 
and thus $\mathrm{Sym}^n(X)$ has HP. 
\end{example}

\begin{theorem}\label{thm:sym_prods_1}
Let $K$ be finitely generated,
$C$  a curve of genus $g_C\geq 2$ over $K$, 
and  $n\in\mathbb{N}$. 
\begin{enumerate}
    \item ${\rm Sym}^n(C)$ has potential WHP if and only if $n\geq g_C$.
    \item $\mathrm{Sym}^n(C\times \mathbb{P}^1_K)$ has potential WHP for every $n\geq g_C$.  
\end{enumerate}
\end{theorem}

\begin{proof}
(1) is \cite[Lemma~6.3]{BJL}.
(2) is \cite[Theorem~B]{BJL},
whose proof uses Theorem \ref{thm:mixed_fibration}, Theorem~\ref{thm:homogeneous} and (1).
\end{proof}

\begin{remark} 
Note that $\mathrm{Sym}^n(C\times \mathbb{P}^1_K)$ for $n\geq g_C$
is special, see \cite[Theorem 14]{CCR},
so Theorem \ref{thm:sym_prods_1} confirmed a case of Conjecture \ref{conj:campana_corvaja_zannier}.
In general, Conjecture \ref{conj:campana_corvaja_zannier}
suggests that 
if $K$ is finitely generated and
$X$ has potential WHP, then 
by Proposition \ref{thm:special} also
any resolution of singularities of ${\rm Sym}^n(X)$ has potential WHP.
This, however, is weaker than the statement that 
${\rm Sym}^n(X)$ itself has potential WHP.
Moreover, it does not apply to 
$X=C$ or
$X=C\times\mathbb{P}^1_K$ as in Theorem~\ref{thm:sym_prods_1},
cf.~Theorem~\ref{thm:curves} and Proposition~\ref{prop:uniruled}.  
\end{remark}

\end{document}